\documentclass[11pt]{amsart}

\usepackage{xcolor}
\pagecolor{white}
\usepackage{amsmath}
\usepackage{amssymb}
\usepackage{amsthm}
\usepackage{hyperref}			
\hypersetup{urlcolor=blue, colorlinks=true}		
\usepackage{geometry}
\usepackage{subfigure}
\usepackage[bf]{caption}
\usepackage{tikz-cd}
\usepackage{cleveref}
\usetikzlibrary{shapes,snakes}
\usepackage{float}
\usepackage{graphicx}
\usepackage{lscape} 
\usepackage[toc,page]{appendix}
\usepackage{mathrsfs}

\usepackage{tikz}
\usetikzlibrary{arrows,decorations.pathmorphing,backgrounds,positioning,fit,petri}
\usetikzlibrary{decorations.pathreplacing}
\usetikzlibrary{decorations.markings}
\usetikzlibrary{decorations.shapes}
\usetikzlibrary{arrows.meta}
\usetikzlibrary{quotes,angles}
\usetikzlibrary{positioning}
\usetikzlibrary{patterns} 
\usetikzlibrary{chains}
\usetikzlibrary{hobby}

\renewcommand{\phi}{\varphi}

\newcommand{\QQ}{\mathbf{Q}}

\newcommand{\ZZ}{\mathbf{Z}}
\newcommand{\RR}{\mathbf{R}}
\newcommand{\CC}{\mathbf{C}}

\newcommand{\iso}{\cong}

\newcommand{\mc}[1]{\mathcal{#1}}

\newcommand{\gl}{\mathrm{GL}}
\newcommand{\GO}{\mathrm{GO}}

\renewcommand{\mod}[1]{\text{\! }(\operatorname{mod}\text{\! }#1)}
\renewcommand{\varepsilon}{\epsilon}

\DeclareMathOperator{\sgn}{sgn}
\DeclareMathOperator{\Gal}{Gal}

\DeclareMathOperator{\tr}{tr}

\DeclareMathOperator{\SL}{SL}
\DeclareMathOperator{\Vol}{Vol}
\DeclareMathOperator{\sh}{sh}
\DeclareMathOperator{\Sh}{Sh}
\DeclareMathOperator{\Gr}{Gr}

\DeclareMathOperator{\Jac}{Jac}

\renewcommand{\Re}{\mathrm{Re}}

\newtheorem{theorem}{Theorem}[section]

\newtheorem{proposition}[theorem]{Proposition}
\newtheorem{lemma}[theorem]{Lemma}
\newtheorem{corollary}[theorem]{Corollary}
\newtheorem{definition}[theorem]{Definition}

\newtheorem{example}[theorem]{Example}

\numberwithin{equation}{section}

\newtheorem{innercustomthm}{Theorem}
\newenvironment{customthm}[1]
  {\renewcommand\theinnercustomthm{#1}\innercustomthm}
  {\endinnercustomthm}

\theoremstyle{remark}
\newtheorem{remark}[theorem]{Remark}

\newcommand{\OK}{\mc{O}_K}

\usetikzlibrary{arrows}

\newcommand{\Hawaii}{Hawai\kern.005em`\kern.005em\relax i}
\newcommand{\HAWAII}{HAWAI\kern.005em`\kern.005em\relax I }

\allowdisplaybreaks

\begin{document}

\title[On the shape of pure prime degree number fields]{O\MakeLowercase{n the shapes of pure prime degree number fields}}

\author{Erik Holmes}
\address{
Department of Mathematics\\
University of Calgary\\
Calgary, AB\\
Canada
}
\email{erik.holmes@ucalgary.ca}
\date{\today}

\begin{abstract}

For $p$ prime and $\ell = \frac{p-1}{2}$, we show that the shapes of pure prime degree number fields lie on one of two $\ell$-dimensional subspaces of the space of shapes, and which of the two subspaces is dictated by whether or not $p$ ramifies wildly. When the fields are ordered by absolute discriminant we show that the shapes are equidistributed, in a regularized sense, on these subspaces. We also show that the shape is a complete invariant within the family of pure prime degree fields. This extends the results of Harron, in \cite{purecubics}, who studied shapes in the case of pure cubic number fields. Furthermore we translate the statements of pure prime degree number fields to statements about Frobenius number fields, $F_p = C_p\rtimes C_{p-1}$, with a fixed resolvent field. Specifically we show that this study is equivalent to the study of $F_p$-number fields with fixed resolvent field $\QQ(\zeta_p)$.

\end{abstract}


\maketitle


\section{Introduction}

The shape of a number field is an Archimedean invariant of the field. 
Roughly speaking, the shape of a degree $n$ number field is the equivalence class of rank $n-1$ lattices obtained from $K$, considered up to scaling, rotation, and reflection. We will define this more precisely in the following section, but first we discuss some previous work in the area and explain a bit of the motivation for this study.

\subsection{Motivation}

In \cite{terr} Terr studied the shape in the case of cubic number fields where he showed that the shapes were equidistributed in the space of rank 2 lattices. This proves that the shapes are, in some sense, random when we consider all cubic number fields and this `randomness' result was later extended to quartic and quintic number fields by Manjul Bhargava and Piper H, in \cite{manjulpiper}. To achieve this generalization, the authors used Bhargava's parametrization of quartic and quintic rings given in \cite{bharIII, bharIV} to show that the shape of quartic and quintic number fields are equidistributed in the space of rank 3 and 4 lattices. The authors conjecture that the same equidistribution result should hold for $n>5$ but, without a nice parametrization of such rings, this is where this line of work with generic fields stops. The authors do, however, mention the following:
\vspace{.07in}

\begin{center} \textit{ ``It is an interesting problem to determine the distribution of lattice shapes for n-ic number fields having a given non-generic (i.e., non-$S_n$) associated Galois group, even heuristically."} \end{center}

\vspace{.07in}

The distribution of non-generic fields has witnessed a number of recent studies in families of field with low degree. We will list many of these studies in the next few sections, highlighting also the study of shape as an invariant of a number field. The purpose of this paper is to generalize the results of Harron on the shapes of pure cubic fields, in \cite{purecubics}, to all pure prime degree fields.

\subsection{Our results}
In this paper we study the shapes, and distribution, of pure prime degree number fields: that is, number fields obtained by adjoining a root of the \textbf{pure prime degree polynomial} $f(x) = x^p - m$, with $m$ $p$-th power free. In this short section we state our results, how they compare to those in \cite{purecubics}, and in the following sections we discuss how these results fit into the recent work in this area. 

Let $\mathcal{S}_{p-1}$ denote the space of shapes of rank $(p-1)$-lattices, and $\ell= \frac{p-1}{2}$. Our first result shows that the shape of pure prime degree fields lie on one of two $\ell$-dimensional subspaces of $\mathcal{S}_{p-1}$, depending on the ramification of $p$ in the field (wild vs tame). We will, as in \cite{purecubics}, denote those pure prime degree $p$ fields where $p$ is wildly (resp. tamely) ramified by Type I (resp. Type II). 
The following theorem shows that the shape of Type I fields are orthorhombic lattices and the shape of Type II fields are ordinary lattices. To state this result we interpret the space of orthorhombic lattices coming from Type I fields as lying on an orbit of the diagonal torus $\mathcal{T}\subseteq \SL_{p-1}(\RR)$. Letting $\mathcal{G}_{wild}({\bf 1})$ be the Gram matrix of the ``square" lattice, represented by the $(p-1)\times (p-1)$ identity matrix, we define $\mathscr{S}_{I}$ to be the orbit of $\mathcal{G}_{wild}({\bf 1})$, under the action of the torus $\mathcal{T}$. For Type II fields things are less uniform: we define the space $\mathscr{S}_{II}$ to be an $\ell$-dimensional subspace of ordinary lattices: see remark \ref{tameconjugation} for more on this. Then we have the following:

\begin{customthm}{A}[Space of shapes]\label{shapespace}
	The shape of a pure prime degree field $K$, denoted $\sh(K)$, lies on one of two $\ell$-dimensional subspaces of $\mathcal{S}_{p-1}$:
	\begin{itemize}
		\item If $K$ is of Type I then $\sh(K) \in \mathscr{S}_I$; i.e. the shape of Type I fields are orthorhombic. 
		\item If $K$ is tamely ramified then  $\sh(K) \in \mathscr{S}_{II}$; i.e. the shape of Type II fields are \textbf{ordinary} lattices as in \cite{terr}. 
	\end{itemize}
		
\end{customthm}

When $p=3$ we are studying the shape of certain rank 2 lattices which are represented by points in the upper half plane modulo the action of $\gl_2(\ZZ)$, see \cref{rank2shapes} for a visualization of this space. Define two subspaces of $\mathcal{S}_{2}$ by $\mathscr{S}_I(3) = \{ z\in \mathcal{S}_2: \Re(z) = 0\}$ and $\mathscr{S}_{II}(3) = \{z\in \mathcal{S}_2 : \Re(z) = \frac{1}{3}\}$. A result of Harron then shows

\begin{customthm}{a}[\cite{purecubics}, Space of shapes]\label{harronpurecubicspace}
	The shape of a pure cubic number field lies on one of two vertical geodesics in $\mathcal{S}_2$:
	\begin{itemize}
		\item If $K$ is of Type I then the shape of $K$ lies on $\mathscr{S}_I(3)$; i.e. the shapes of Type I fields are rectangular. 
		\item  If $K$ is of Type II then the shape of $K$ lies on $\mathscr{S}_{II}(3)$; i.e. the shapes of Type II fields are parallelograms with no extra symmetry, or \textbf{ordinary} lattices.
	\end{itemize}
		
\end{customthm}

To compare Harron's result with the generalization we note that $\mathscr{S}_I(3)$ can be described as the orbit of $y=i$ under the action of the torus:
	\[	\mathcal{T} = \begin{pmatrix} t & 0 \\ 0 & t^{-1} \end{pmatrix}.	\]

Our second result concerns the invariance of shapes and shows that the shape is a complete invariant within the family of all pure, prime degree, number fields. 

\begin{customthm}{B}[Invariance of shape]\label{invariance}
	Given $K$ and $L$, two pure prime degree number fields, we have that $K\cong L$ if and only if $\sh(K) = \sh(L)$. 
\end{customthm}

\begin{remark} 
Harron proved this in the case of pure cubic fields however he also showed that more is true in this case: combining the results of \cite{purecubics} and \cite{complexcubics}, he showed that the shape is a complete invariant within the family of all complex cubic fields. Our generalization shows that the complete invariance of shape holds within the family of pure prime degree fields but this is just one step towards a potentially stronger statement which is the subject of future research.
\end{remark}

The third statement of the paper concerns the distribution of shapes: in particular we prove that the shapes are equidistributed (in a regularized sense) along $\mathscr{S}_I$ (resp. $\mathscr{S}_{II}$) with respect to the measure $\mu_I$ (resp. $\mu_{II}$).

\begin{customthm}{C}\label{equidistribution}
	For $q$ prime, let $\delta_q(\mathcal{L}(Y))$ be a $q$-adic density, and define constants:
	\begin{align*} 
		C_{wild}  	&	= \frac{\pm1}{(2p-1)2^{p-2}p^{\ell-1}h_p^-}\prod_{q} \delta_q(\mathcal{L}_n(Y))	\\
		C_{tame} 	&	= \frac{\pm (2p-2)}{(2p-1)2^{p-2}p^{\ell-1}h_p^-}\prod_{q} \delta_q(\mathcal{L}_n(Y)).	\\
	\end{align*}
	Then we have the following:
	\[	\lim_{X\rightarrow \infty} \frac{N_{wild}(X, W)}{X^{1/{p-1}}\log(X)^{\ell-1}} = C_{wild} \mu_{I}(W)	\]
	and 
	\[	\lim_{X\rightarrow \infty} \frac{N_{tame}(X, W)}{X^{1/{p-1}}\log(X)^{\ell-1}} = C_{tame}\mu_{II}(W)	\]
	
	\noindent
	Where $N_{wild}(X,W)$ (resp. $N_{tame}(X,W)$) denote the number of wildly ramified (resp. tamely ramified) pure prime degree number fields with discriminant bounded by $X$ and shape in $W$.  
\end{customthm}

This shows that the shapes of Type I (resp.\;Type II) fields are equidistributed, in a regularized sense, along $\mathscr{S}_I$ (resp. $\mathscr{S}_{II}$). This extends the main result of Harron's pure cubic paper:

\begin{customthm}{c}[\cite{purecubics}, regularized equidistribution]\label{purecubicequidistribution}
	Define constants:
	\[	C_I = \frac{2C\sqrt{3}}{15} \text{ and } C_{II} = \frac{C\sqrt{3}}{10}	\]
	where 
	\[	C = \prod_q \left(1-\frac{3}{q^2} + \frac{2}{q^3} \right)	\]
	and the product is over all primes $q$. For $? = I$, resp. II, and real numbers $1\leq R_1 < R_2$, let $[R_1, R_2)_?$ denote the ``interval" $i[R_1, R_2)$, resp. $(1 + i[R_1, R_2))/3$ in $\mathscr{S}_?$. Then, for all $R_1, R_2$
	\[	\lim_{X\rightarrow \infty} \frac{\#\{ K \text{ of type } ?: |\Delta(K)|\leq X, \sh(K) \in [R_1, R_2)_?\}}{C_? \sqrt{X}} = \int_{[R_1, R_2)_?} d\mu_? \]
	where $\Delta(K)$ is the discriminant of $K$, and $\sh(K)$ is the shape of $K$.
	
\end{customthm}

\begin{remark}
The regularized equidistribution statement in \cref{equidistribution} is the higher dimensional analogue of \cref{purecubicequidistribution}. The usual statement of equidistribution would have, as a denominator, the number of pure prime degree fields with absolute discriminant bound but Benli, in \cite{purefieldasymptotics}, shows that this quantity has more $\log$ terms and would therefore send the limit to $0$. The reader can compare this result to those in \cite{purecubics, galoisquartic, multiquad}: in each there is at least one parameter in the space of shapes that is unrestricted and we can think of these as having infinite length. In \cite{purecubics} this is seen clearly as the geodesics that the shapes lie on are vertical. In this paper we sample shapes in some compact subset of the space which yields logarithmic terms containing shape conditions and not the discriminant bound. Incorporating the shape conditions into the asymptotics error term would be an interesting avenue to explore as it seems that the field asymptotic can be recovered in at least some cases: we discuss this in a bit more detail in \cref{asymptoticintro}. 
\end{remark}

Theorems \ref{shapespace}, \ref{invariance} and \ref{equidistribution} extend the results of \cite{purecubics} to all pure prime degree number fields.
Our final result allows us to interpret this study in terms of Galois conditions and resolvent fields:

\begin{customthm}{D}\label{theoremgalois}
	$K$ is a pure prime degree number field if, and only if, $\Gal(\tilde{K}/\QQ)\iso F_p = C_p\rtimes C_{p-1}$ and $K$ has (degree $p-1$) cyclotomic resolvent field $\QQ(\zeta_p)$. 
\end{customthm}

This result is widely known in the cubic case where the statement is that pure cubics are exactly those ($S_3\iso F_3$) cubic fields whose quadratic resolvent is $\QQ(\zeta_3)$: for a proof in the cubic case see Lemma 33 of \cite{manjulari}. We prove the generalization which not only allows us to phrase things in a manner more similar to the work being done in arithmetic statistics, but also motivates other questions which the author plans to answer in future work, see \cref{frobeniussection} for more on this. 

In the remaining sections of the introduction we attempt to motivate our results by describing how they fit into the recent work in this area.

\subsection{The strength of invariance}

It is well known that the discriminant is a complete invariant of a number field of degree $d$ if, and only if, $d=2$. As the shape is closely related to the discriminant it is natural to ask whether the extra information it contains makes it a stronger invariant or not. 

Terr, in \cite{terr}, proved that cyclic cubic fields all have the same (hexagonal) shape! It seems that the extra symmetries of the field, which are inherited by the shape, can cause too much collision for the shape to be a strong invariant so what else do we know?

\begin{remark}
The study of shapes as invariants has been addressed in a few recent papers and continues to be an active area of study. We compiled a short list of some recent results in this direction. 

\begin{itemize}
	\item Guillermo Mantila-Soler and Marina Monsurr\`o, in \cite{zmodlz}, study the shape of cyclic number fields of prime degree and show that the shape is no more powerful an invariant than the discriminant. 

 	\item  William Bola\~{n}os and Guillermo Mantila-Soler, in \cite{zmodnz}, extend the work above, \cite{zmodlz}, to study the shape of cyclic degree $n$ number fields and show that, again, the shape of $K$ gives you nothing more than the discriminant of the number field. As such, the shape (or trace form in their work), is far from a complete invariant of $K$ in the case of cyclic number fields. 
 
	 \item In contrast to the aforementioned work, Rob Harron has shown, in \cite{purecubics} and \cite{complexcubics}, that the shape is a complete invariant of pure cubic fields and, more generally, of complex cubic fields (just as the discriminant is in the case of quadratic fields). This shows that, given two such number fields $K_1$ and $K_2$, $K_1\iso K_2 \iff \sh(K_1) = \sh(K_2)$. Where we say $\sh(K)$ to mean the shape of the field $K$. 
	 
	 \item Piper H and Rob Harron, in \cite{galoisquartic}, show that if we restrict to certain families of $V_4$ quartics, that the shape again determines the field.
	 
\end{itemize}
\end{remark}
	 
	 These results show that the shape can be a very powerful invariant in some cases, and a very weak invariant in others; we see that Galois fields, and specifically those with cyclic Galois group, with extra symmetries may force shapes to collide but in general much remains to completely classify the shape as an invariant. Theorem \ref{invariance} provides, for each prime $p$, an infinite family of number fields for which the shape is a complete invariant. 

One reason we care about this statement and the power of invariance is towards relating the study of shapes to the study of number field asymptotics. This is discussed in more detail in \ref{asymptoticintro} but the quick idea is that by studying the distribution of shapes in prescribed families we can sometimes obtain the corresponding number fields asymptotics, or at least some information about the $\log$ terms in the asymptotics.

\subsection{Distribution of shapes}

As mentioned above Terr studied the question of non-generic distribution in the case of cyclic cubic number fields, where all fields have the same shape. This result is argued geometrically using the fact that there is only one point in the space of rank 2 lattices which admits an order 3 automorphism. Terr's results show how Galois conditions impose strong restrictions on the shape of the field and where said shape lies: in one case (the generic, $S_3$, case) the shapes distribute randomly in the entire space of shapes of rank 2 lattices while the other case (the non-generic, $C_3$, case) yields a single point in the space of shapes. This should be somewhat intuitive based on the symmetries that the lattice inherits from the field but how this changes, when the space of shapes has higher dimension and the families that we consider are not cyclic, continues to be an intriguing question.

\begin{remark} Here are a few of the recent studies of shape distributions.
\noindent
\begin{itemize}
	\item Rob Harron studied complex cubics, in \cite{purecubics} and \cite{complexcubics}, and showed that the shapes of these cubics lie, and are equidistributed, on geodesics defined by the quadratic resolvent of the cubic field. In the case of pure cubic fields the geodesics that the shapes lie on are vertical and have infinite length, whereas the non-pure complex cubics lie on geodesics of finite length. This witnesses the difference in the asymptotics of cubic fields with prescribed quadratic resolvents as is discussed in the following section. 
	\item Piper H and Rob Harron study the shapes of Galois quartic extensions, in \cite{galoisquartic}, showing that the shape of $V_4$ quartics are equidistributed along subspaces of the space of shapes, and that cyclic quartic fields are not. Specifically the cyclic quartic shapes distribute discretely along subspaces of the space of shapes and are therefore not equidistributed in any sense. The latter result can be compared with Terr's cyclic cubic result, though the higher dimensional space of shapes give more options for lattices with cyclic symmetry. 
	\item Jamal Hassan, in \cite{multiquad}, extended the $V_4$ result above by studying the shapes of octic multiquadratic extensions and showing that the shapes are equidistributed along subspaces of the space of shapes. 
\end{itemize}
\end{remark}

Our current work extends the results in \cite{purecubics} to all pure prime degree number fields (i.e. fields of the form $K = \QQ(\sqrt[p]{m}$)). Specifically, \cref{equidistribution} shows that the shapes of these fields equidistribute (in a regularized sense) along one of two subspaces of the space of (rank $p-1$) shapes and which space they lie on depends on the ramification of $p$ in the field.

With the goal of showing that the shapes of fields in a given family are equidistributed we need to first determine the shapes as we vary over the fields in our family. Once the shapes have been determined we will create a parametrization, or bijection, between the shapes and certain integer points in some bounded subset of a real vector space. The idea is then to approximate the number of integer points in the region using the principle of Lipschitz and sieve methods from analytic number theory. Though some of the methods of this paper coincide with the ones used in \cite{purecubics} we will ultimately end up with a new proof of Harron's pure cubic results while also extending them to all pure prime degree number fields.

\subsection{A Galois theoretic interpretation of pure fields and number field asymptotics}\label{asymptoticintro}

This final section provides a bit of an aside to the current project but will allow us to phrase the results in a manner that may be more familiar to anyone studying number field asymptotics. We have two major motivations in the study of shapes which stem from Malle's conjecture: the first is that it seems number field asymptotics can often be obtained from shape studies, and the second in relation to this is that the shape seems to explain $\log$ terms in certain cases. We hope to motivate both phenomena in this section.  

\begin{remark}
We first talk about the shapes connection to $\log$ terms in Malle's conjecture.
\begin{itemize}
	\item In \cite{purecubics, complexcubics} Harron observed this phenomena. In these cases the shapes are 2-dimensional lattices which, after appropriate scaling/rotation, can be viewed as points in the upper half plane. These points lie on geodesics in $\mathcal{H}$ determined by the fields trace zero form and these geodesics have finite length\footnote{with respect to the hyperbolic measure on $\mathcal{H}$} for non-pure complex cubic fields. This difference coincides with the asymptotics of cubic fields with fixed quadratic resolvent: e.g. let $N_3(X, F_2)$ be the number of cubic fields with quadratic resolvent $F_2$ and absolute discriminant bounded by $X$. The following result was proved independently by Henri Cohen and Anna Morra, in \cite{CohenMorra}, and Manjul Bhargava and Ari Shnidman, in \cite{manjulari}:
	\begin{align}\label{cubicresolvent}
		N_3\left(X, \QQ\left(\sqrt{d}\right)\right) \sim \begin{cases} X^{1/2} & d\not= -3	\\	X^{1/2}\log(X) & d = -3	\end{cases}	
	\end{align}
When $d=-3$ the quadratic field contains the cube roots of unity and the corresponding cubic field is pure. So, we see $\log$ terms in the asymptotics when the geodesics have infinite measure. 

	\item In \cite{galoisquartic} Piper H and Rob Harron show a similar phenomena happens in the case of $V_4$ quartic fields. If $N_4(X, V_4)$ is the number of $V_4$-quartic fields with discriminant bounded by $X$ then $N_4(X, V_4) \sim X^{1/2}\log(X)$. The authors mention that if the shape of such fields lie in a `box', of side length $R$, which constrains the (two) shape parameters then the number of fields with discrimnant bounded by $X$ and shape in this box grows like $X^{1/2}\log^2(R)$. 
	
	\item In forthcoming work the author and Rob Harron, \cite{harronholmes}, study the shape of non-Galois sextic fields (i.e. those sextic fields with absolute Galois group $C_3\wr C_2$) with a fixed quadratic subfield. Jurgen Kl\"uners, in \cite{Kluners}, showed that these fields witness the failure of Malle's conjecture which predicts the correct power of $X$ in the asymptotics but not a $\log$ term that appears when you count those sextic fields with quadratic subfield $F_2=\QQ(\sqrt{-3})$\footnote{Again, this is the cyclotomic field $\QQ(\zeta_3)$.}. In fact more is true: if you count those non-Galois sextic fields which don't have this cyclotomic subfield then the count does adhere to Malle's conjecture. In two cases we study those fields with cyclotomic subfield, and (to make things explicit) those whose subfield has class number 1, to show that the `shape' of these fields all lie on one dimensional subspaces of the space of shapes. Furthermore, the subspaces on which they lie have infinite measure when the quadratic subfield is $\QQ(\zeta_3)$, and finite measure otherwise. Comparing this with the above studies we see logarithmic terms appearing in the first case and not in the second. 
\end{itemize}
In all cases we observe a relationship between the space in which the shapes lie and the asymptotics (specifically the $\log$ terms in the asymptotics) of the fields. 

\end{remark}

The works in the remark above motivated this section and that in the final section of that paper. We wanted to see if the same phenomena, as in \cite{purecubics}, occurs in this generalization. That required a Galois interpretation, a fixed resolvent field and the asymptotics for fields with both the given Galois conditions and the fixed resolvent: theorem \ref{theoremgalois} provides most of this by proving that the pure prime degree fields are exactly those degree $p$ number fields whose Galois group\footnote{Here we mean the Galois group of the Galois closure} is isomorphic to the Frobenius group, and whose (unique) degree $(p-1)$ resolvent field is isomorphic to $\QQ(\zeta_p)$. We define all of this precisely in the final section but for now we note that $F_p$ is called the Frobienius group and is defined, for our purposes, as the semi direct product:
	\[	F_p = \left(\ZZ/p\ZZ\right)\rtimes \left(\ZZ/p\ZZ\right)^\times	\]
and this provides a Galois theoretic interpretation of pure prime degree fields. 

Now, for the asymptotics of such fields: K\"{u}bra Benli, in \cite{purefieldasymptotics}, proves that the number of pure prime degree fields with discriminant bounded by $X$ grows like $X^{1/p} Q_p(\log(X))$ where $Q_p(x)$ is a polynomial of degree $p-2$. Using \cref{theoremgalois} this gives asymptotic results for the number of (degree $p$) Frobenius fields with fixed cyclotomic resolvent field. And with that we can phrase our results, as was done in \cite{galoisquartic}, as follows: the number of fields with discriminant bounded by $X$, and shape in a `hypercube' of side length $R$ which constrains the ($\ell$) shape parameters, grows like $X^{1/p}\log(X)^{\ell-1} H(\log(R))$ where $H(x)$ is a homogeneous polynomial of degree $\ell$. When $p=3$ it is easy to incorporate the shape conditions into the error and recover the asymptotic in \ref{cubicresolvent}. This completes a part of the generalization we are after: the remaining part, and subject of another project, is to repeat this study with Frobenius fields with different fixed resolvent fields.

The main takeaway from this is that we can phrase our study in terms of Galois groups and resolvent fields rather than simply as pure extensions of $\QQ$. This aligns the study of shapes and their distributions a little more closely with work in the direction of Malle's conjecture which has, in general, motivated a number of the authors projects. In general there is much to do to relate shape studies to number field asymptotics but, as is stated in \cite{complexcubics}, it would be interesting to have heuristics for the field counting functions that incorporate shape conditions.

\tableofcontents


\section{The shape of a number field}

The shape of a number field is defined to be the equivalence class of the lattice $j(\mathcal{O}_K^\perp)$ up to scaling, rotation, and reflection. To define this precisely we need to know what the shape of an arbitrary lattice is, and then to specify how we obtain the lattice $j(\mathcal{O}_K^\perp)$. 

\subsection{The shape of a lattice}
We let $\Lambda$ be a rank $r$ lattice. The \emph{shape} of $\Lambda$ is the equivalence class of this lattice up to scaling, rotation, and reflection. Often in the literature the space of shapes of rank $r$ lattices is presented as the double coset space:
	\[	S_r = \gl_r(\ZZ)\backslash \gl_r(\RR) / \GO_r(\RR)	\]
where $\gl$ denotes the general linear group and $\GO$ the group of orthogonal matrices\footnote{i.e. those matrices $M$ such that $M\cdot M^T = I$.}. We take an alternative approach, as in \cite{galoisquartic}, which will instead define the space of shapes in terms of Gram matrices. Letting $\mathcal{G}$ denote the set of positive definite real symmetric $(r\times r)$-matrices we have that the space of shapes of rank $r$ lattices is also represented as
	\[	S_r = \gl_r(\ZZ)\backslash \mathcal{G}/\RR^\times.	\]
Given a lattice $\Lambda$ in an inner product space $V$ and with basis $\{b_1, \hdots, b_r\}$ we form the Gram matrix of $\lambda$, denoted $\Gr(\Lambda)$, by taking innerproducts of the basis:
	\[	\Gr(\Lambda) = (\langle b_i, b_j\rangle)_{1\leq i,j\leq r}.	\]
This is an element of $\mathcal{G}$ and uniquely identifies the lattice up to change of basis. For $g\in \gl_r(\ZZ)$ and $G\in \mathcal{G}$ we define the left action of $\gl_r(\ZZ)\backslash \mathcal{G}$ as: 
	\[	g\cdot G := gGg^T	\]
which comes from the natural action on the set of Gram matrices. For $r\in \RR^\times$ and $G\in \mathcal{G}$ we define the right action of $\mathcal{G}/\RR^\times$ as:
	\[	G\cdot r := r^2 G	\]
which, again, comes from the natural action of scaling the basis vectors of $\Lambda$ and computing $\Gr(\Lambda)$. When we refer to the shape of a lattice we will often be referring to an explicit Gram matrix whose calculation we describe in the next section. For further details we refer the reader to \cite{galoisquartic}.

\subsection{The lattice $j(\mathcal{O}_K^\perp)$}
Let $K$ be a number field of degree $n$, $\OK$ be it's ring of integers with basis $\{1, \alpha_1, \hdots, \alpha_{n-1}\}$, and $\{\sigma_i\}_{1\leq i \leq n}$ be the set of embeddings of $K$ into $\CC$. Then we have the Minkowski embedding:
	\begin{align*}
		j: K		&	 \rightarrow \CC^n	\\
		\alpha	&	\mapsto (\sigma_1(\alpha), \hdots, \sigma_n(\alpha)).		
	\end{align*}
The $\RR$-span of the image of this map is an inner-product space, often denoted by $K_\RR \cong \RR^n$, and the restriction of $j$ to $\OK$ yields a rank $n$ lattice in $K_\RR$. Though there is some temptation to define the shape of $K$ to be the shape of this lattice there are issues when investigating the distribution of shapes as we vary over some family of fields: namely, as all such lattices can be made to have the common vector $j(1)$, we lose the potential ``randomness'' of shape.  As such, we define the shape of $K$ to be the shape of the lattice obtained by projecting $j(\OK)$ onto the orthogonal complement of $j(1)$: specifically we define the perp map
	\begin{align*}
		\alpha^\perp := n\alpha - \tr(\alpha)
	\end{align*}
and let $\OK^\perp$ be the image of $\OK$ under this map. Note that the elements of $\OK^\perp$ are elements of trace-zero and map, under the Minkowski map, to vectors orthogonal to $j(1)$. Using this we obtain a lattice, $j(\OK^\perp)$, of rank $n-1$ and we define the shape of $K$ to be the shape of this lattice. 
Given a basis of $\OK$, $\{1, \alpha_1, \hdots, \alpha_{n-1}\}$, we obtain a basis of $\OK^\perp$, $\{\alpha_1^\perp, \hdots, \alpha_{n-1}^\perp\}$, and we obtain a representative of the shape by computing the Gram matrix:
	\[	\Gr(j(\OK^\perp)) = (\langle j(\alpha_i^\perp), j(\alpha_j^\perp)\rangle)	.	\]


\section{Shapes of pure prime degree number fields}
One will notice, in the studies mentioned in the introduction, that the shape of $K$ often depends on the ramification of $p$ in $K$. More specifically the subspace on which the shape of $K$ lies is determined by whether $p$ is wildly ramified in $K$ or not: this is referred to as a wild versus tame dichotomy in \cite{purecubics, galoisquartic}. The same phenomenon occurs in this study and so, after we determine a sufficiently nice basis for $\OK$, we will split this section up according to whether $p$ is wildly ramified in $K$, or not: our notation for such $K$ will be $K_{\textup{wild}}$ and $K_{\textup{tame}}$ respectively, unless it is otherwise clear from context. 

We assume throughout that $p$ is a prime number greater than $2$; this will simplify the statements of lemmas \ref{puredisc} and \ref{purepbasis} but, as all quadratic fields have the same shape, we will not be losing anything with this simplification.

\subsection{ The integral basis of $K/\QQ$ }

Let $K = \QQ(\alpha)$ where $\alpha$ is the unique real root\footnote{This gives us that $K\subseteq \RR$ which will simplify notation a bit in what follows.} of the irreducible polynomial $x^p-m$ belonging to $\ZZ[x]$ and $m= \prod_{i=1}^{p-1} a_i^i$ where the $a_i$ are squarefree and pairwise relatively prime. Throughout this paper we will refer to $K$, defined in this way, as a {\bf pure prime degree number field}.
The following two lemmas, from \cite{purepbasisnew}, provide us with the discriminant and an integral basis for $K/\QQ$ respectively:

\begin{lemma}[\cite{purepbasisnew}]\label{puredisc}
	Let $K=\QQ(\alpha)$ be a pure prime degree number field, as above, and let $q$ run over all distinct primes dividing $m$. We have the following:
	\begin{enumerate}
		\item[(i)] When $p\mid m$, or $p\nmid m$ and $p^2\nmid (m^p-m)$, the discriminant of $K$ is given by:
			\[ \Delta_K = \left(-1\right)^{\frac{(p-1)}{2}} p^p \prod_{q| m} q^{p-1}. \]
		\item[(ii)]  If $p^2\mid (m^{p-1}-1)$, then the discriminant of $K$ is given by:
			\[ \Delta_K = \left(-1\right)^{\frac{(p-1)}{2}} p^{p-2} \prod_{q|m} q^{p-1}.\]
	\end{enumerate}
\end{lemma}

As $p$ divides the discriminant of $K$ we know that $p$ is ramified in both cases: when $p^p\mid \Delta(K)$ we say $p$ is wildly ramified in $K$. Otherwise, we say that $p$ is tamely ramified in $K$. Using this we have that the fields in (i) are wild, and the fields in (ii) are tame. The next lemma provides an integral basis in each case:

\begin{lemma}[\cite{purepbasisnew}]\label{purepbasis} Let $K, \alpha, m$ be defined as above and, for $j\in\{1, \hdots, p-1\}$, let $\gamma_j = \alpha^j/\prod_{i=1}^{p-1} a_i^{\lfloor \frac{ij}{p}\rfloor}$. Then we have:
	\begin{enumerate}
		\item[(i)] If $p^2\nmid (m^{p-1}-1)$, then the set $\{ 1, \gamma_1, \gamma_2, \hdots, \gamma_{p-1}\}$ is an integral basis of $K_{\textup{wild}}$.
		\item[(ii)] If $p^2\mid (m^{p-1}-1)$, then the set $\{ \gamma_0, \gamma_1, \gamma_2, \hdots, \gamma_{p-1}\}$ is an integral basis of $K_{\textup{tame}}$ where $\gamma_0 = \frac{1}{p}\sum_{i=0}^{p-1}(\epsilon\alpha)^i$ where $\epsilon$ is any integer so that $m\epsilon \equiv 1\mod{p^2}$.
	\end{enumerate}
\end{lemma}

To provide a familiar example and to highlight a bit of the notation to follow we remind ourselves that:
\begin{example}\label{purecubicbasis}
When $p=3$ and $K=\QQ(\sqrt[3]{m})$ we write $m = a_1a_2^2$, $\alpha = \sqrt[3]{m}$ and $\beta = \sqrt[3]{a_1^2a_2}$. If $3^2 \mid m^2 - 1$ then $m\equiv \pm 1 \mod{9}$ and $\epsilon$ is $\pm 1$. The integral basis of $K$ is  then given by:
\[
	 \begin{cases} 		\left\{ 1, \alpha, \frac{\alpha^2}{a_2}\right\} = \left\{ 1, \alpha,\beta\right\} & \text{ when } m \not\equiv \pm 1 \mod 9 \\ 
					\left\{ \frac{1\pm \alpha + \alpha^2}{3}, \alpha, \frac{\alpha^2}{a_2}\right\} = \left\{ \frac{1\pm \alpha + \alpha^2}{3}, \alpha, \beta\right\} & \text{ when } m\equiv \pm 1 \mod 9. \end{cases}
\]

Note that the congruence condition on $m$ is the same as the condition on ramification mentioned above. In \cite{purecubics} we see that the shapes of pure cubic fields lie on $2$ vertical geodesics in the space of rank $2$ lattices:  moreover the ramification at $3$ determines which vertical geodesic the shape lies on. We will see exactly what this all looks like below but those who are too excited to wait can look at figure $\ref{rank2shapes}$!
\end{example}

\subsection{ Slightly altered tame basis} We note that the integral basis of $K_{\textup{tame}}$ given in the previous section will not be particularly nice for our purposes: the shape of $K$ is given by mapping an integral basis into $\CC^n$ and projecting onto the orthogonal subspace spanned by $j(1)$ and because of this it is nice to write an integral basis which contains the number 1. To that end we choose our basis to be $\{1, \nu, \gamma_2, \hdots, \gamma_{p-1}\}$ where 	
	\[	\nu = \frac{m + \alpha + \epsilon\alpha^2 + \hdots + \epsilon^{p-2}\alpha^{p-1}}{p} \]
That this element, $\nu$, is integral comes from the fact that $\epsilon m\equiv 1 \mod{p^2}$. Let $k\in \ZZ$ be such that $\epsilon m = 1+ kp^2$. Then we have:
	\begin{align*}
		m\gamma_0 - kp\sum_{i=1}^{p-1}\alpha^i 	&	=	m\left(\frac{1+\epsilon\alpha + (\epsilon\alpha)^2 + \cdots + (\epsilon\alpha)^{p-1}}{p}\right) - kp\sum_{i=1}^{p-1}\alpha^i	\\
										&	=	\frac{m + \sum_{i=1}^{p-1}(\epsilon m - kp^2)\epsilon^{i-1}\alpha^i}{p}	\\
										&	=	\frac{m + \alpha + \epsilon\alpha^2 + \cdots + \epsilon^{p-2}\alpha^{p-1}}{p}	\\
										&	=	\nu.
	\end{align*}
	
Thus, $\nu\in \OK$ and the change of basis matrix from the rational basis $\{1, \gamma_1, \gamma_2, \hdots, \gamma_{p-1}\}$  to the basis $\{1, \nu, \gamma_2, \hdots, \gamma_{p-1}\}$ is:
	\[	
		\mathcal{C}_{t} = 
		\begin{pmatrix}
			1			&	0			&	0						&	\hdots	&	\hdots	&	0	\\
			\frac{m}{p}		&	\frac{1}{p}		&	\frac{\epsilon b_2}{p}			&	\hdots 	&	\hdots	&	\frac{\epsilon^{p-2}b_{p-1}}{p}	\\
			0			&	0			&	1						&	0		&	\hdots 	&	0	\\
			\vdots		&	\ddots		&	\ddots					&	\ddots	&	\ddots	&	\vdots	\\
			0			&	\hdots		&	\hdots					&	\hdots	&	0		&	1
			
		\end{pmatrix}
	\]
where $b_j = \prod_{i=1}^{p-1} a_i^{\lfloor\frac{ij}{p}\rfloor}$. The determinant of this matrix is $\frac{1}{p}$ which shows that the basis $\{1, \nu, \gamma_2, \hdots, \gamma_{p-1}\}$ has the correct discriminant\footnote{Square and compare to \cref{puredisc}.} and we can therefore conclude that it is indeed an integral basis; we use this as our basis of $K_{\textup{tame}}$ below. 

\subsection{ The shape of $K_{\textup{wild}}$} 

We let $\zeta_p = e^{2\pi i/p}$ and $\alpha$ be as above. We also let $\sigma$ denote the real embedding of $K$ and $\tau_k$ denote the complex embedding sending $\alpha$ to $\zeta_p^k\alpha$ for $k \in \{1, \hdots, \frac{(p-1)}{2}\}$. Finally we let $j$ be the embedding of $K$ into $\CC^p$ given by:
	\[ j(a) = (\sigma(a), \tau_1(a), \hdots, \tau_{(p-1)/2}(a), \overline{\tau}_{(p-1)/2}(a), \hdots, \overline{\tau_1}(a)).\] 
Note that we have changed the ordering of $j$ from above, this does not change anything but we feel this is a slightly more pleasant way of presenting it.
 
Using the above embedding, $j$, we have the following result about the Gram matrix of the basis for pure prime degree \textit{\textup{wild}} fields:
\begin{proposition}\label{wildgram}
	The Gram matrix of the basis $\{1, \gamma_1, \gamma_2, \hdots, \gamma_{p-1}\}$ is given by:
		\[ 	
			\mathcal{G}(K_{\textup{wild}}) = \begin{pmatrix}	
				p	&				&				&			&					\\
					&	p\gamma_1^2	&				&			&					\\
					&				&	p\gamma_2^2	&			&					\\
					&				&				&	\ddots	&					\\
					&				&				&			&	p\gamma_{p-1}^2	\\
			\end{pmatrix}
		\]

\end{proposition}

\begin{proof}
	To see this we first note that
		\begin{align*}
			j(1) 			&	= (1, \hdots, 1)	\\
			j(\gamma_i) 	&	= (\gamma_i, \zeta_p^i\gamma_i, \zeta_p^{2i}\gamma_i, \hdots, \zeta_p^{(p-1)i}\gamma_i)
		\end{align*}
\noindent	
Taking the Hermitian inner product of all pairs of vectors we find that the $(i,j)$ entry of the Gram matrix, $\mathcal{G}_{i,j}$, is given by:
		
		\begin{align*}	
			\sum_{r=0}^{p-1} \zeta_p^{ir}\gamma_i\cdot \overline{\zeta_p^{jr} \gamma_j} 	
					&	= \gamma_i\gamma_j \sum_{r=0}^{p-1} \zeta_p^{ir}\overline{\zeta_p^{jr}} \\
					&	= \gamma_i\gamma_j \sum_{r=0}^{p-1} \zeta_p^{ir}\zeta_p^{(p-j)r} \\
					&	= \gamma_i\gamma_j \sum_{r=0}^{p-1} \zeta_p^{(p+i-j)r} 
		\end{align*}
\noindent		
When $i\not= j$, $p+i-j \equiv k \mod{p}$ for some $k\in \{1, \hdots, p-1\}$ and therefore $\sum_{r=0}^{p-1} \zeta_p^{(k)r} = \sum_{n=0}^{p-1} \zeta^n = 0$. This shows that all off diagonal entries are 0, and when $i=j$ we have $\mathcal{G}_{i,i} = \gamma_i^2 \sum_{r=0}^{p-1} \zeta_p^{(p+i-i)r} = \gamma_i^2 \sum_{r=0}^{p-1} 1 = \gamma_i^2\cdot p$ as desired. 
	\end{proof}

From the Gram matrix in \cref{wildgram} we obtain the shape of $K_{\textup{wild}}$ by noting that the vectors $j(\gamma_1), \hdots, j(\gamma_{p-1})$ span $j\left(\mathcal{O}_K^\perp\right)$:

\begin{theorem}\label{wildshape}
	The shape of $K_{\textup{wild}}$ is represented by:
		\[
			\begin{pmatrix}	
				\gamma_1^2	&		0		&	\hdots	&		0			\\
					0		&	\gamma_2^2	&			&		\vdots		\\
					\vdots	&				&	\ddots	&		0			\\
					0		&		\hdots	&		0	&	\gamma_{p-1}^2	\\
			\end{pmatrix}
		\]
\end{theorem}

From this we see that the shape of $K_{\textup{wild}}$ is an orthorhombic lattice: i.e. $\mathcal{O}_K^\perp$ is spanned by $p-1$ pairwise orthogonal vectors of length $|\gamma_i|$. When $p=3$ we are again in the pure cubic case and we see that, when $3$ is wildly ramified, the shape of $K$ is rectangular (orthorhombic) and lies on the \textcolor{blue}{blue}  geodesic in $S_2$ shown as follows:

	\begin{figure}[H]
	\centering
		\begin{tikzpicture}[scale = 3.5]
			\tikzset{myptr/.style={decoration={markings,mark=at position 1 with %
    							{\arrow[scale=3,>=stealth]{>}}},postaction={decorate}}}
							
			\draw[-] (0,0)--(0,2) ;
			\draw[->, >=stealth] (-.25,0)--(1,0);
		
			\begin{scope}
    				\clip (0.007,.86) rectangle (.5,2.01);
    				\shade[
					bottom color = gray,
					middle color = white,
					top color = white,
					shading angle = 0
					] (0,1) -- (.5,.866) -- (.5,2) -- (0,2);
					
				\draw[thick] (.5,.866) -- (.5,2.01);
				\filldraw[fill=white, draw=black] (0,0) circle (1);
				\draw[draw=white, pattern=north west lines, opacity=.2, pattern color=gray] (0,1) to[bend left=14] (.5, .866) -- (.5,2.1) -- (0,2.1) -- (0,1);
				
			\end{scope} 
			\draw[line width=.25mm,  ->, >=stealth, blue, opacity=.7] (0,1)--(0,2.05);
			\draw[line width=.25mm,  ->, >=stealth, red, opacity=.7] (1/3,0.94280903)--(1/3,2.05);
			\node[scale=.75] at (.25, 1.5) {$S_2$};
		\end{tikzpicture}
		\caption{Subspaces of shapes of pure cubic fields (\textcolor{blue}{wild} and \textcolor{red}{tame}) within the space of shapes of rank 2 lattices, $S_2$}
		\label{rank2shapes}
	\end{figure}
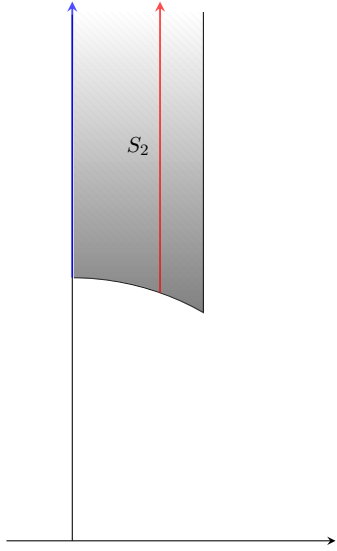

\noindent	
When $K$ is tame the shape is not orthorhombic and the distribution is along the \textcolor{red}{red} geodesic in $S_2$ with real part $\frac{1}{3}$. 
	
\begin{example}
 Using Theorem \ref{wildshape} in the cubic example above, \cref{purecubicbasis}, for $K$ wildly ramified at 3, we see that the shape of K is
	\[	\sh(K) = \begin{pmatrix} 	
			\sqrt[3]{a_1^2a_2^4}	&	0	\\
			0				&	\sqrt[3]{a_1^4 a_2^2}
		\end{pmatrix}.
	\]	
Though this is a rectangular lattice with side lengths $\sqrt[3]{a_1a_2^2}$ and $\sqrt[3]{a_1^2a_2}$ it is more convenient to scale this matrix by the (1,1)-entry giving us:
	\[	\sh(K) = 	\begin{pmatrix}
					1	&	0	\\
					0	&	\sqrt[3]{\frac{a_1}{a_2}}^2	
				\end{pmatrix}.
	\]	
which represents a rectangular lattice who side lengths are $1$ and, in the notation of \cite{purecubics}, $r_K^{1/3} = \sqrt[3]{\frac{a_1}{a_2}}$. Without loss of generality we may assume $a_1>a_2$ and we see that this lattice can be represented as a point on the blue geodesic above with height determined by $r_K$. 
\end{example}

We note that it is common to scale the Gram matrix of $\mathcal{O}_K^\perp$ to obtain a unit vector, as above, and to call this the shape of $K$. What we do below is slightly different but will better illuminate the beautiful symmetry in these fields we study. The reader may also notice that, in section 4, this alternative scaling, and hence alternative shape, will make the inherited measure feel very intuitive.

\subsection{Shape of $K_{\textup{tame}}$}
We saw above that an integral basis of $K_{\textup{tame}}$ is given by $\{1, \nu, \gamma_2, \hdots, \gamma_{p-1}\}$ where 
	\[	\nu = \frac{m + \alpha + \epsilon\alpha^2 + \hdots + \epsilon^{p-2}\alpha^{p-1}}{p}.	\]
We also saw the change of basis matrix from $\{1, \gamma_1, \gamma_2, \hdots, \gamma_{p-1}\}$ to the basis above is given by:
	\begin{align}\label{conjugatematrix}
		\mathcal{C}_{t} = 
		\begin{pmatrix}
			1			&	0			&	0						&	\hdots	&	\hdots	&	0	\\
			\frac{m}{p}		&	\frac{1}{p}		&	\frac{\epsilon b_2}{p}		&	\hdots 	&	\hdots	&	\frac{\epsilon^{p-2}b_{p-1}}{p}	\\
			0			&	0			&	1						&	0		&	\hdots 	&	0	\\
			\vdots		&	\ddots		&	\ddots					&	\ddots	&	\ddots	&	\vdots	\\
			0			&	\hdots		&	\hdots					&	\hdots	&	0		&	1
		\end{pmatrix}.
	\end{align}
The Gram matrix of $K_{\textup{tame}}$ is therefore given by
	\begin{align}\label{change}	 \mathcal{G}(K_{\textup{tame}}) = \mathcal{C}_{t} \mathcal{G}(K) \mathcal{C}_{t}^T	\end{align}
where $\mathcal{G}(K)$ is the Gram matrix of the rational basis, i.e.~the Gram matrix of $K_{\textup{wild}}$ seen above.

\begin{proposition}
	The Gram matrix of the integral basis of $K_{\textup{tame}}$ given above is:
	\[ 	\mathcal{G}(K_{\textup{tame}}) = \begin{pmatrix}
			p		&	m										&	0								&	0								&	\hdots	&	0								\\
			m		&	\nu'										&	\epsilon(\gamma_2)(\gamma_1)^2		&	\epsilon^2(\gamma_3)(\gamma_1)^3		&	\hdots	&	\epsilon^{p-2}(\gamma_{p-1})(\gamma_1)^{p-1}	\\
			0		&	\epsilon(\gamma_2)(\gamma_1)^2				&	p\gamma_2^2						&	0								&	\hdots	&	0								\\
			\vdots	&	\epsilon^2(\gamma_3)(\gamma_1)^3				&	0								&	p\gamma_3^2						&	\ddots	&	\vdots							\\
					&	\vdots									&									&	\ddots							&	\ddots	&	0								\\
			0		&	\epsilon^{p-2}(\gamma_{p-1})(\gamma_1)^{p-1}		&	\hdots						&									&	0		&	p\gamma_{p-1}^2					\\
		\end{pmatrix}
	\]
	where $\nu' = \frac{1}{p}\left(m^2 + \gamma_1^2 + \epsilon^2\gamma_1^4 + \hdots + \epsilon^{2p-4}\gamma_1^{2p-2}\right)$.
\end{proposition}	

\begin{proof}
	This is shown by computing \ref{change} and noting that $b_i\gamma_i^2 = b_i \gamma_i \left(\frac{\gamma_1^i}{b_i}\right) = \gamma_i \gamma_1^i$. 
\end{proof}

Using this we obtain the shape of $K_{\textup{tame}}$ as follows:

\begin{theorem}\label{tameshape}
	The shape of $K_{\textup{tame}}$ is represented by:
	\[
		\begin{pmatrix}
				\nu'-\frac{m^2}{p}						&	\epsilon(\gamma_2)(\gamma_1)^2	&	\epsilon^2(\gamma_3)(\gamma_1)^3		&	\hdots	&	\epsilon^{p-2}(\gamma_{p-1})(\gamma_1)^{p-1}		\\
				\epsilon(\gamma_2)(\gamma_1)^2			&	p\gamma_2^2					&	0								&	\hdots	&	0							\\
				\epsilon^2(\gamma_3)(\gamma_1)^3			&			0					&	p\gamma_3^2						&	\ddots	&	\vdots						\\
					\vdots							&			\vdots				&		\ddots						&	\ddots	&	0							\\
				\epsilon^{p-2}(\gamma_{p-1})(\gamma_1)^{p-1}	&			0					&			\hdots					&		0	&	p\gamma_{p-1}^2				\\
		\end{pmatrix}
	\]
\end{theorem}

\begin{proof}
	Note that $j(\gamma_i)$ is orthogonal to $j(1)$ for all $i$. In the case of $j(\nu)$ we have 
		\[	j(\nu)^\perp = j(\nu) - \frac{\langle j(\nu),j(1)\rangle}{\langle j(1), j(1)\rangle} j(1) = j(\nu) - \frac{m}{p}j(1).\] 
	Therefore 
		\begin{align*}
			\langle j(\nu)^\perp, j(\nu)^\perp\rangle 	&	= \langle j(\nu), j(\nu)\rangle - 2 \left\langle j(\nu), \frac{m}{p} j(1)\right\rangle + \left\langle \frac{m}{p} j(1), \frac{m}{p} j(1)\right\rangle \\
											&	= \langle j(\nu), j(\nu)\rangle - \frac{2m}{p} \langle j(\nu), j(1)\rangle + 	\frac{m^2}{p^2}\langle j(1), j(1)\rangle		\\
											&	= \nu' - \frac{2 m}{p} \cdot m + \frac{m^2}{p^2}\cdot p 	\\
											&	= \nu' - \frac{m^2}{p}		
		\end{align*} 
\end{proof}

In conclusion we proved that the lattice $j(\mathcal{O}_K^\perp)$, for $K$ a pure prime degree number field, is an orthorhombic lattice if $p$ was wildly ramified in $K$ and a general (non-orthorhombic) lattice\footnote{or an ordinary lattice, as in \cite{terr}}, if $p$ was tamely ramified in $K$. Of course the tamely ramified fields are nearly orthogonal and the non-orthogonality comes from a single vector, the remaining sublattice is orthorhombic. Nonetheless, this again gives evidence that the type of lattice (and hence the space in which the shape lies) is determined by the ramification of $p$ in the field $K$. To compare with \cite{purecubics} this is the type I vs type II phenomena whose shapes lie on the imaginary axis and the vertical geodesic with real part $r = \frac{1}{3}$ respectively.


\section{Parametrizations}

The goal of this section will be to develop a parametrization of pure prime degree number fields with specified shape conditions. We first set up a correspondence between the pure prime degree fields and tuples in $\ZZ^{p-1}$ which satisfying some specific conditions. Once we have parametrized the fields, without shape conditions, we will use this to obtain parameterizations of the shapes from these tuples. 

\subsection{Fields} 
To parametrize the family of pure prime degree fields we use {\bf strongly carefree tuples}:

\begin{definition}
	We define a tuple $(a_1, a_2, \hdots, a_n)\in \ZZ^{n}$ to be {\bf strongly carefree} if $a_i$ is squarefree for all $i$ and $(a_i, a_j) = 1$ for all $j\not=i$ (i.e. the $a_i$ are squarefree and pairwise relatively prime). 
\end{definition}

\noindent
Note that since $m^{1/p}$ and $(-m)^{1/p}$ generate the same pure prime degree number field we need only consider the tuples in which $a_i>0$ for all $i$. For sake of uniformity we denote the set of strongly carefree (and positive) $n$-tuples by $\mathcal{SC}^n$ or, when the degree of $K$ is clear, by $\mathcal{SC}$ as in \cite{galoisquartic}. 

The pure prime degree fields are exactly parametrized by strongly carefree $(p-1)$-tuples. For two strongly carefree tuples $(a_1, a_2, \hdots, a_{p-1})$ and $(a_1', a_2', \hdots, a_{p-1}')$ we obtain integers $m=a_1a_2^2\cdots a_{p-1}^{p-1}$ and $m'=a_1'(a_2')^2\cdots (a_{p-1}')^{p-1}$. Furthermore $m$ and $m'$ generate the same field if $m' =  \gamma^{p}m^i$ for $(i, p)=1$ and $\gamma\in \ZZ$. Define an action of $g\in S_{p-1}$ on the set of strongly carefree tuples by:
	\[	g\cdot (a_1, a_2, \hdots, a_{p-1}) = (a_{g(1)}, a_{g(2)}, \hdots, a_{g(p-1)}).	\]
We can then define the group
\[	
C_{p-1} = \left\langle 	
			\begin{pmatrix}	
				1 			& 	2  	& 	3			&	4 	& 	 \hdots		& p-4 	& 	p-3			& 	p-2 	& 	p-1 \\	 
				\frac{p-1}{2}	&	p-1	&	\frac{p-1}{2} -1	&	p-2	&	 \hdots		& 2	&	\frac{p+1}{2}+1	&	1	& 	\frac{p+1}{2}
			\end{pmatrix}\right\rangle.
\]
which gives us the following:
\begin{theorem}
	There is a bijection:
	\begin{align*}
		\left\{ \text{\begin{minipage}{.4\textwidth}Isomorphism classes of pure prime degree, $p$, number fields \end{minipage}} \right\} && \longleftrightarrow  && \left\{ \text{\begin{minipage}{.3\textwidth} $C_{p-1}$-\text{equivalence classes of } $(a_1, \hdots, a_{p-1})\in \mathcal{SC}^{p-1} $\end{minipage}}\right\}.
	\end{align*}
\end{theorem}

\begin{example}
	When $p=5$ we have $K = \QQ(\sqrt[5]{m})$ with $m= a_1a_2^2a_3^3a_4^4$ and $C_4 = \langle (1243)\rangle$. The elements of the rational basis, as in \cref{purepbasis}, are
	\begin{align*}
		\gamma_1 &	= (a_1a_2^2a_3^3a_4^4)^{1/5}	&
		\gamma_2 &	= (a_1^2a_2^4a_3a_4^3)^{1/5}	&
		\gamma_3 &	= (a_1^3a_2a_3^4a_4^3)^{1/5}	&
		\gamma_4 &	= (a_1^4a_2^3a_3^2a_4)^{1/5}	
	\end{align*}
	and the orbit of $(a_1, a_2, a_3, a_4)$ under $g=(1243)$ is
		\[
			\left\{(a_1, a_2, a_3, a_4),  (a_2, a_4, a_1, a_3), (a_4, a_3, a_2, a_1), (a_3, a_1, a_4, a_2)\right\}.
		\]
\end{example}

\begin{remark}
	Letting $\mathcal{N}_p(X)$ be the number of pure, prime degree, fields with absolute discriminant bounded by $X$ we could obtain an asymptotic for $\mathcal{N}_p(X)$ by counting the number of strongly carefree tuples satisfying $\prod a_i < X^{1/p}$ and $a_i\geq 1, \forall i$. In what follows we will want to count those fields with absolute discriminant bounded by $X$ and shape in some ``nice" set $W$: we denote the fields of bounded discriminant and shape in $W$ by $\mathcal{N}_p(X, W)$. To count this we need to impose the shape conditions which will impose additional restrictions on the $a_i$. 
\end{remark}

\subsection{Shapes}
Though we will state results in the case of tame fields we will work primarily with wild fields in what follows recalling the simple conjugation between the two in \ref{conjugatematrix}. Now, the shape of $K$ is defined to be the lattice $j(\mathcal{O}_K^\perp)$ up to scaling, rotation, and reflections and so we will be a bit more specific in our representative of the shape of $K$. Sepecifically we will choose a representative of the shape of wild fields which has hypervolume 1 and whose diagonal entries are given in increasing order. The following normalization will show us the symmetry present in these lattices which forces the shapes to lie on an $\ell$-dimensional subspace of $S_{p-1}$. It will also help us prove the statement of regularized equidistribution and show the shape is a complete invariant of pure prime degree number fields.

\subsubsection{Shape parameterization}
Let $\ell = \frac{p-1}{2}$. One may have noticed that in the Gram matrix of $j(\mathcal{O}_K)$, and hence the shape matrix, a dependence on the $\gamma_j$: namely we see that $\gamma_j =\frac{\gamma_\ell\gamma_{\ell+1}}{\gamma_{p-j}}$. The determinant of the matrix given in Theorem \ref{wildshape} is then
	\[	(\gamma_\ell\gamma_{\ell+1})^{\ell} = \prod_{i=1}^{p-1} a_i^{\ell}. \]
Using this we are lead to a very natural normalization of the shape by $(\gamma_{\ell}\gamma_{\ell+1})^{-1} = \prod_i a_i^{-1}$. This yields the following representative of shape:
	\begin{align}\label{scaleshapewild}
			\sh(K_{\textup{wild}}) =
			\begin{pmatrix}	
				\lambda_1		&				&					&							&			&					\\
							&	\ddots		&					&							&			&					\\
							&				&	\lambda_{\ell}		&							&			&					\\
							&				&					&	\lambda_{\ell}^{-1}		 	&			&					\\
							&				&					&							&	\ddots	&					\\
							&				&					&							&			&	\lambda_1^{-1}	\\
			\end{pmatrix}
	\end{align}
where $\lambda_i = \frac{\gamma_i^2}{\gamma_{\ell}{\gamma_{\ell+1}}}$ for $1\leq i \leq \ell$ and we immediately see the $\ell$ parameters that govern the shape of these fields. 

Similarly, after normalizing by the same factor, the shape of $K_{\textup{tame}}$ given in Theorem \ref{tameshape} is:
	\begin{align}\label{scaleshapetame}
		\sh(K_{\textup{tame}})  = \begin{pmatrix}
				\nu''			&	\hdots		&	\epsilon^{\ell-1} b_\ell \lambda_\ell	&	\frac{\epsilon^{\ell} b_\ell}{ \lambda_\ell}	&	\hdots	&	 \frac{\epsilon^{p-2}b_{p-1}}{\lambda_1}	\\
							&	\ddots		&	0							&	\hdots							&			&	0								\\
							&				&	p\lambda_\ell					&									&			&	\vdots							\\
							&				&								&	\frac{p}{\lambda_\ell}					&			&									\\
							&				&								&									&	\ddots	&	0								\\
							&				&								&									&			&	\frac{p}{\lambda_1}					\\
		\end{pmatrix}
	\end{align}
where 
	\[	\nu'' = \frac{1}{p} \left( \sum_{i=1}^{\ell} \left( \epsilon^{2(i-1)}b_i  \lambda_i^2\right) + \left(\epsilon^{2(p-i-1)}b_{p-i}\frac{1}{\lambda_i^2}\right) \right).	\]

In both cases we see very explicitly that the shape of $K$ lies in an $\ell$-dimensional subspace of the space of rank $p-1$ lattices, and that the parameters controlling the shape are given by $\{\lambda_i\}_{i=1}^{\ell}$.

\begin{example}\label{shapeparametersex5} For $p=5$, $K = \QQ\left(\sqrt[5]{a_1a_2^2a_3^3a_4^4}\right)$, and $p$ wildly ramified in $K$, Theorem $\ref{wildshape}$ tells us that:
	\[	\Gr(\mathcal{O}_K^\perp) 
			= 	\begin{pmatrix} 
					\left(a_1^2a_2^4a_3^6a_4^8\right)^{1/5}		&	0								&	0									&	0						\\
								0						&	\left(a_1^4a_2^8a_3^2a_4^6\right)^{1/5}	&	0									&	0					\\
								0						&				0					&	\left(a_1^6a_2^2a_3^8a_4^4\right)^{1/5}		&	0						\\
								0						&				0					&			0							&	\left(a_1^8a_2^6a_3^4a_4^2\right)^{1/5}	
				\end{pmatrix}	
	\]
	The aforementioned normalization by $\prod_i a_i^{-1}$ gives us the following shape parameters:
	\begin{align*}
		\lambda_1 = \left(\frac{a_4^3a_3}{a_1^3a_2}\right)^{1/5} 	&&	\lambda_2 = \left(\frac{a_4a_2^3}{a_1a_3^3}\right)^{1/5}	\\
	\end{align*}

\end{example}

Note that as the shape of a number field is defined up to scaling, rotation, and reflection, it makes sense to scale as we have done and then perform some change of basis to obtain a shape matrix with increasing diagonal entries. After doing this we define the shape of $K$ as:

\begin{definition}
	We choose a representative of the shape of a pure prime degree number field, denoted $\Sh(K)$, to be the matrix $\sh(K)$ above, where the diagonal elements are given in increasing order. 

\noindent
Or, in terms of the shape parameters, we have that: 
	\[	\Sh(K) = \{(\lambda_{\sigma_1}, \lambda_{\sigma_2}, \hdots, \lambda_{\sigma_\ell}): 1< \lambda_{\sigma_1}<\lambda_{\sigma_2}<\hdots<\lambda_{\sigma_\ell}\}.	\]
\end{definition}

We now describe the shape parameters explicitly in terms of the $a_i$ and end this section by showing that the shape of a pure prime degree field, $K$, is a complete invariant. Note that the $\{\lambda_i\}$ and the $\{\lambda_{\sigma_i}\}$ are the same and so our description below, while not necessarily ordered as in $\Sh(K)$, will be enough for our purposes. 

\begin{lemma}\label{shapeparameters}
	The shape parameters of $K$, $\{\lambda_i^{\pm1}\}_{i=1}^{\ell}$, are given by:
		\[	\lambda_j = \left(\prod_{i=1}^{\ell} \left( \frac{a_i}{a_{p-i}}     \right)^{2ij - p - 2p\left\lfloor \frac{ij}{p}\right\rfloor}	\right)^{1/p}.	\]
\end{lemma}

\begin{proof} Since $\lambda_j = \frac{\gamma_j^2}{\gamma_\ell \gamma_{\ell + 1}}$ we have:
	\begin{align*}
		\lambda_j^p 		&	= \left(\frac{\gamma_j^2}{\gamma_{\ell} \gamma_{\ell +1}} \right)^p\\
						&	= \frac{\prod_{i=1}^\ell a_i^{2ij - 2p\lfloor\frac{ij}{p}\rfloor} a_{p-i}^{2j(p-i) - 2p\lfloor\frac{(p-i)j}{p}\rfloor}}{\prod_{i=1}^{p-1} a_i^p }\\
						& 	= \prod_{i=1}^\ell a_i^{2ij - 2p\lfloor\frac{ij}{p}\rfloor - p} a_{p-i}^{2jp-2ji - 2p\lfloor j - \frac{ij}{p}\rfloor - p} \\
						& 	= \prod_{i=1}^\ell a_i^{2ij - 2p\lfloor\frac{ij}{p}\rfloor - p} a_{p-i}^{2jp-2ji - 2p(j - 1 - \lfloor\frac{ij}{p}\rfloor) - p} \\
						& 	= \prod_{i=1}^\ell a_i^{2ij - 2p\lfloor\frac{ij}{p}\rfloor - p} a_{p-i}^{2jp-2ji - 2pj +2p + 2p\lfloor\frac{ij}{p}\rfloor - p} \\
						& 	= \prod_{i=1}^\ell a_i^{2ij - p - 2p\lfloor\frac{ij}{p}\rfloor} a_{p-i}^{-2ji  + p + 2p\lfloor\frac{ij}{p}\rfloor} \\
						& 	=  \prod_{i=1}^{\ell} \left( \frac{a_i}{a_{p-i}}     \right)^{2ij - p - 2p\left\lfloor \frac{ij}{p}\right\rfloor}.
	\end{align*}
\end{proof}

\begin{remark}
Given the shape of a pure prime degree number field we get an element in the isomorphism class of $K$ by simply considering $\QQ(\lambda_j)$. More precisely the shape parameters, 
	\begin{align*}
		\lambda_j 		&	= \frac{\gamma_j^2}{\gamma_{\ell-1}\gamma_{\ell+1}} \\
					&	= \frac{\gamma_j^2}{\prod_{i=1}^{p-1} a_i}				\\
					&	= \frac{\left(\frac{\alpha^{2j}}{\prod_{i=1}^{p-1}a_i^{\left\lfloor\frac{ij}{p}\right\rfloor}}\right)}{\prod_{i=1}^{p-1} a_i}	\\
					& 	= \frac{\alpha^{2j}}{\prod_{i=1}^{p-1} a_i^{1 + \left\lfloor\frac{ij}{p}\right\rfloor}}
	\end{align*}
lie in $K = \QQ(\alpha^{2j})=\QQ(\alpha)$\footnote{This is because $\gcd(j, p)=1$}. Lemma \ref{shapeparameters} shows that the shape parameters are not rational numbers and therefore that $\QQ\subsetneq \QQ(\lambda_j)$. As $K$ is of prime degree and therefore has no intermediate subfields and we conclude that $\QQ\subsetneq \QQ(\lambda_j) \subset \QQ(\alpha^{2})= \QQ(\alpha)$ or that $\QQ(\lambda_j)=\QQ(\alpha)$. 

\end{remark}

It is then clear that two non-isomorphic fields have different shapes: this is enough to prove Theorem \ref{invariance} and show that the shape is a complete invariant within the family of pure prime degree number fields. One may ask whether this holds within the family of degree $p$ number fields? For instance, is the fact that these shape parameters are pure numbers (i.e. those coming exactly from pure prime degree fields) enough for shape to discriminate between prime degree fields?  The reader can compare this with the result of Harron in \cite{purecubics}. 

We end this section with an example:

\begin{example}
	As with the example above we let $K = \QQ\left(\sqrt[5]{a_1a_2^2a_3^3a_4^4}\right)$, and assume that $K$ is wild: then the shape parameters of $K$ are: 
	\begin{align*}
		\lambda_1 = \left(\frac{a_4^3a_3}{a_1^3a_2}\right)^{1/5} 	&&	\lambda_2 = \left(\frac{a_4a_2^3}{a_1a_3^3}\right)^{1/5}.	
	\end{align*}
	It is clear that $\QQ(\lambda_1)\subseteq \QQ(\alpha^2)$ as $\lambda_1 = \frac{\alpha^2/a_4}{a_1a_2a_3a_4}$ and that $\lambda_i\not\in \QQ$. Since there are no intermediate subfields we conclude that $\QQ(\lambda_1) = K$ and hence the shape determines the pure quintic field. 
\end{example}


\section{Measure Theoretic Background... towards equidistribution}\label{measuresection}
This section is devoted to the measures on the spaces of shapes of wild (resp. tame) pure prime degree number fields.
In general for a locally compact (Hausdorff) topological group, $G$, there exists a unique\footnote{up to scaling} Haar measure, $\mu_G := \mu$, which is left $G$-invariant (i.e.~$\mu(gA) = \mu(A)$ for all $A\subset G$ and $g\in G$; where $gA = \{gx\mid x\in A\}\subset G$).

\subsection{Measure on spaces of shapes}
Recall that, for $K$ a pure prime degree field with wild ramification at $p$ we have that the shape is a complete invariant of $K$ and that the shape lies in an $\ell$-dimensional subspace of the space of shapes. The shape parameters are given by the $\{\lambda_j\}$ in Lemma \ref{shapeparameters}. We define the measure on the space of shapes of each family of fields to be:
		\begin{align} 	d\mu_?(x_i) := \prod_i \frac{dx_i}{x_i}		\label{haarmeasure}		\end{align}
where $?= I$ if $K$ is wild (resp. $? = II$ if $K$ is tame) and $dx_i$ is the usual Lebesgue measure on $\RR$. This notation is motivated by that in \cite{purecubics} and the wild (type I) vs tame (type II) phenomena we see in both of these studies.

\begin{remark}
Similar to the situation in \cite{purecubics, galoisquartic} we see that this measure comes from the measure on certain diagonal matrices. First, note that the set of gram matrices:
	\[	G_{w}(\{x_i\}) 	= 	\begin{pmatrix}
						x_1	& 			&		&			&			&			\\
							&	\ddots	&		&			&			&			\\
							&			&	x_d	&			&			&			\\
							&			&		&	x_d^{-1}	&			&			\\
							&			&		&			&	\ddots	&			\\
							&			&		&			&			&	x_1^{-1}	\\
					\end{pmatrix}
	\]
is an orbit of the group $\mathcal{T}\subseteq\SL_{p-1}(\RR)$ where:
	\[	\mathcal{T} := \{G_{w}(x_i): x_i\in \RR^\times\}\footnote{$G(\{\sqrt{x_i}\})\cdot G(\{1\}) = G(\{x_i\})$}. 	\]
The measure on this group is exactly the one we see above which induces a measure on space of shapes of wild fields. 
\end{remark}

\begin{remark}\label{tameconjugation}
We saw in \ref{conjugatematrix} that the set of Gram matrices of tame fields is obtained by conjugating the wild ones by $\mathcal{C}_t$: specifically we have $\mathcal{G}_{tame} = \mathcal{C}_t \mathcal{G}_{wild} \mathcal{C}_t^T$. If we could get rid of the dependency of $\mathcal{C}_t$ on the field and find a uniform $\mathcal{C}_t$ with this property then we would get that the set of tame Gram matrices is an orbit of the conjugated torus $\mathcal{C}_t \mathcal{T}\mathcal{C}_t^{-1}$.

Then the shape of tame fields would inherit the Harr measure from the conjugated torus $C_t\mathcal{T}C_t^{-1}$. However, we do know that the shape is parametrized by the same $\ell$ shape parameters, and we define the measure on these matrices to be the measure defined in \ref{haarmeasure}. 
\end{remark}

We now follow a similar exposition to that in \cite{galoisquartic}, recalling that equidistribution is a statement about the weak convergence of a sequence of measures. Note that a sequence of measures $\{\mu_N\}$ on $S_{I}$ converges weakly to $\mu_{I}$ if for all compactly supported continuous functions, $f\in \mathcal{C}_c(S_{I})$, 
	\[	\lim_{N\rightarrow \infty} \int_{S_{I}} f d\mu_N = \int_{S_{I}} d\mu_{I}.	\]
The measure above, $\mu_N$, is defined by:

\[	\mu_N(W) = \frac{\#\{K \in \mathcal{K}_{\textup{wild}} : |\Delta(K)| \leq N, \Sh(K)\in W\}}{C N \log^{\ell -1}(N)}.		\]

For $x_i\in \RR$ we let
	\[	\Sh_{I}(\{x_i\}) = 
			\begin{pmatrix}  
				x_1		&	0		&			&	\hdots 		&			&	0	\\
				0		&	\ddots	&			&				&			&		\\
						&			&	x_\ell		&	\ddots		&			&	\vdots	\\
				\vdots	&			&	\ddots	&	x_{\ell}^{-1}	&			&					\\
						&			&			&				&	\ddots	&	0		\\
				0		&	\hdots	&			&	\hdots		&	0		&	x_1^{-1}
			\end{pmatrix} \in S_{I}.
	\]

Also, for $R_i\in \RR$ we define
	\[	W_{I}(R_1, R_2, \hdots, R_\ell) = W_{I}(\{R_i\}) = \{ \Sh_{I}(x_i): R_1\leq x_1< x_2<\hdots< R_{\ell+1}, \text{ and } R_i<x_i	\}	\] 
and let $\chi_{I\{R_i\}}$ denote its characteristic function, we will show that it is sufficient to test these functions towards showing equidistribution. This is an immediate generalization of Lemma 3.10 in \cite{galoisquartic} and is also included in \cite{multiquad} so the reader can see either of these references for more details. 

\begin{lemma}\label{wildmeasure}
	If we have that
	\[	\lim_{N\rightarrow \infty} \int_{S_{I}} \chi_{I \{R_i\}} d\mu_N		= \int_{S_{I}} \chi_{I \{R_i\}} d\mu_{I}	\]
for all $\{R_i\in \RR_{\geq 1}: R_i\leq R_{i+1} \text{ for all } i\in \{1, \hdots, \ell-1\}\}$ then $\mu_N$ converges weakly to $\mu_{I}$.
\end{lemma}

\begin{proof}
We recall that compactly supported continuous functions on $\RR^\ell$ can be approximated above and below by step functions of hypercubes or, in our case, by step functions of sets $W_{I}(\{R_i\})$. The proof of this follows from the methods in \cite{purecubics} and repeated applications of inclusion and exclusion to show that hypercubes in $\RR^\ell$ can be written in terms of the sets $W_{I}(\{R_i\})$. For full details we refer the reader to Proposition 2.46 of \cite{multiquad}.

\end{proof}

\begin{lemma}\label{homogeneous}
	For $\{R_i\}$ as above we have that:
		\[	\mu_{I}(W_{I}(\{R_i\})) = \frac{1}{\ell!}H(\{R_i\})	\]
	where $H(\{R_i\})$ is a homogeneous polynomial of degree $\ell$ in $\log(R_i)$. 
\end{lemma}

\begin{proof}
	This follows from the fact that
		\begin{align*}
			\mu_{I}(W_{I}(\{R_i\}))		&	=	\int_{R_{\ell}}^{R_{\ell+1}} \int_{R_{\ell}}^{x_\ell}\cdots \int_{R_1}^{x_2} \frac{1}{x_1\cdots x_\ell} dx_1\cdots dx_\ell 
			\hfill \qedhere
		\end{align*}
\end{proof}

We will see this measure explicitly in the case of $p=3$ and $p=5$ and so we end this section with a calculation of the measure in each case:

\begin{example}
For $p=3$, $\ell=1$, and we get that
	\[	\mu_I(W_I(\{R_1, R_2\})) = \int_{R_1}^{R_2} \frac{dx_1}{x_1} = \log\left(\frac{R_2}{R_1}\right)	\]
\end{example}

\begin{example}
For $p=5$, $\ell=2$, and we get that 
	\[ \mu_{I}(W_{I},\{R_1, R_2, R_3\})	= \frac{1}{2}\left(\log^2\left(\frac{R_3}{R_1}\right)- \log^2\left(\frac{R_2}{R_1}\right)\right)	\]
\end{example}


\section{Equidistribution of shapes}

This section is dedicated to the proof of Theorem \ref{equidistribution} where we will show that the shapes of pure prime degree fields are equidistributed (in a regularized sense) along  $\ell$-dimensional subspaces of the space of shapes. We have already found a parametrization for the fields with prescribed shape condition and we can approximate this number by counting particular lattice points in the region we obtain. Specifically we use the principle of Lipschitz and a strongly-carefree sieve. Let
		$\mathcal{R}(N, R_1, \hdots, R_{\ell+1})$ be
		\[  \left\{ (a_1, \hdots, a_{p-1})\in \RR_{> 0}^{p-1}: \prod_{i=1}^{p-1} a_i < N,  R_1 \leq \lambda_1^p  < \hdots < \lambda_\ell^p  \leq R_{\ell+1} \text{ and } \forall i,  R_i < \lambda_i^p \right\}\]
\noindent
where the $\lambda_i$ are the shape parameters defined in \cref{shapeparameters}. The first step toward counting strongly carefree tuples in this region is to compute the volume of $\mathcal{R}(N, R_1, R_2, \hdots, R_{\ell+1})$ and of its lower-dimensional shadows.

\subsection{Volumes} 

Before we compute any volumes we define a shadow of this region: 

\begin{definition}
	For a prime $p\in \ZZ$ and natural number $d\in\{1, 2, \hdots, p-2\}$ we define a {\bf $d$-dimensional shadow}\footnote{This definition is motivated by Lemma 9 of \cite{bhar05}. } of the region $\mathcal{R}(N, R_1, \hdots, R_{\ell+1})$ to be the orthogonal projection of $\mathcal{R}$ onto any $d$-dimensional coordinate subspace.
\end{definition}

We refer to these projections as the lower-dimensional shadows of $\mathcal{R}(N, \{R_i\})$ or $\mathcal{R}$ for simplicity. In order to apply the principle of Lipschitz we need both the volume of the main region, $\mathcal{R}$, and a bound on the measure of the lower dimensional shadows of $\mathcal{R}$. Towards this we have our first statement:

\begin{proposition}\label{volumeprop}
	For any prime $p$, the volume of $\mathcal{R}(N, R_1, R_2, \hdots, R_{\ell+1})$ is 
		\[	c(N-1)\log^{\ell-1}(N)\cdot H(R_1, \hdots, R_{\ell+1})	\]
	where $H(R_1, \hdots, R_{\ell+1})$ is the homogeneous polynomial in \cref{homogeneous}. 
\end{proposition}

\begin{proof}
To make the volume above a bit more pleasant to compute we define a change of variables from $\{a_i\}$ to $\{x_i\}$:
	\[	x_i	 = 	\begin{cases}	
					\lambda_i^p	&	 \text{ for } i\in \{1, \hdots, \ell\}		\\	
					a_{i-\ell}\cdot a_{p-i+\ell}	&	 \text{ for } i\in \{\ell+1, \hdots, p-2\}	\\ 
				\end{cases}	\]
and
	\[	x_{p-1} = \frac{\prod_{i=1}^{p-1}a_i}{\prod_{i=1}^{p-2} x_i}	\]

\noindent
Note that the Jacobian of this change of variables is given by
		\begin{align}\label{jacobiandet}	J = \begin{pmatrix}
					\frac{\partial x_1}{\partial a_1} 	& 	\hdots		&	\frac{\partial x_1}{\partial a_{p-1}}		\\
					\vdots					&	\ddots		&	\vdots							\\
					\frac{\partial x_{p-1}}{\partial a_1} 	& 	\hdots	&	\frac{\partial x_{p-1}}{\partial a_{p-1}}
				\end{pmatrix}
		\end{align}
whose determinant is equal, by \ref{jacobiandeterminant}, to $c_p = \pm 2^{p-2}p^{\ell-1}h_p^-$, where $h_p^-$ is the negative part of the class number of the cyclotomic field of degree $p$. 

Now, we have explicit bounds on the first $\ell$ variables coming from the definition of $\mathcal{R}$ and so to bound the remaining $\ell$ variables we simply note that the product of these variables is bounded between 1 and $N$:
	\[	1<x_1\cdots x_\ell \cdot x_{\ell+1} \cdots x_{p-1} = a_1\cdot a_2\cdots a_{p-1} < N.	\]
This gives us the following:
\begin{align*}
															& \int_{\mathcal{R}(X, R_1, \hdots, R_{p-1})} da_1\hdots da_{p-1}		\\
														=	&		\int_{R_{\ell}}^{R_{\ell+1}}\int_{R_{\ell-1}}^{x_\ell} \hdots \int_{R_{1}}^{x_2}\int_{1/\prod_{i=1}^{\ell} x_i}^{N/\prod_{i=1}^{\ell} x_i}\hdots\int_{\frac{1}{x_1\cdot x_2\cdots x_{p-2}}}^{ \frac{N}{x_1\cdot x_2\cdots x_{p-2}}}   \frac{1}{c_p}dx_{p-1}\hdots dx_{\ell+1}dx_{1}\hdots dx_{\ell}		\\
														=	&		\frac{1}{c_p}\int_{R_{\ell -1}}^{R_\ell}\int_{R_{\ell-1}}^{x_\ell} \hdots \int_{1/\prod_{i=1}^{p-3} x_i}^{N/\prod_{i=1}^{p-3} x_i}	\frac{N-1}{x_1\cdot x_2\cdots x_{p-2}}	dx_{p-2}\hdots dx_{\ell}\\
														= 	&		\frac{(N-1)\log(N)}{c_p} \int_{R_{\ell}}^{R_{\ell+1}}\int_{R_{\ell-1}}^{x_\ell} \hdots \int_{1/\prod_{i=1}^{p-4} x_i}^{N/\prod_{i=1}^{p-4} x_i}	\frac{1}{x_1\cdot x_2\cdots x_{p-2}}	dx_{p-3}\hdots dx_{\ell}\\
															&		\hspace{2in}\vdots	\\
														= 	&		\frac{(N-1)\log^{\ell-1}(N)}{c_p} \int_{R_{\ell }}^{R_{\ell+1}}\int_{R_{\ell-1}}^{x_\ell} \hdots \int_{R_1}^{x_2} \frac{1}{x_1\cdots x_\ell}	dx_1\cdots dx_\ell	\\
\end{align*}

Where the latter integral is the measure, $\mu_?$, whose evaluation is a homogeneous polynomial of degree $\ell$ in the $\{R_i\}$, see \ref{homogeneous}. 
\end{proof}

Now that we know the volume of the region $\mathcal{R}(N, \{R_i\})$ we need a bound on the volume of the lower dimensional shadows:

\begin{proposition}\label{shadowprop}
	The maximum measure of the lower dimensional shadows of $\mathcal{R}(N, \{R_i\})$ is $\mathcal{O}(N\log^{\ell-2}(N))$.
\end{proposition}

First we note that the bounds on the shape parameters can be used to bound the ratios:
	\[	r_i := \frac{a_i}{a_{p-i}}.	\]
Specifically, for all $i\in \{1, \hdots, \ell\}$, we have that $1\leq R_1<\lambda_i <R_{\ell+1}$ and we can use this to bound each ratio by something of the form:
	\[	R_1^{\delta_i} < r_i^{u_i} <  R_{\ell+1}^{\epsilon_i}	\]
where $u_i\in \{-1, 1\}$ and $\delta_i, \epsilon_i$ are non-zero rational numbers. 

\begin{remark}
It is not hard to see that the bounds on the shape parameters impose bounds on these ratios: if, for all $i\in \{1,\hdots, \ell\}$, we let $v_i = \log\left(\frac{a_i}{a_{p-i}}\right)$ then we are making a statement about the linear independence of the shape parameters, after applying the logarithm. More explicitly we have a linear map from the $\{v_i\}$ to $\{\log(\lambda_j)\}$ which is represented by the matrix:
	\[	C_\ell = \left(\left(2ij - p - 2p\left\lfloor \frac{ij}{p}\right\rfloor \right)\right)_{1\leq i,j \leq \ell}	\]
the invertibility of which is a consequence of \cref{jacobiandeterminant}. This shows that from the bounds on the shape parameters we obtain bounds on the ratios, $r_i$. 
\end{remark}

\begin{remark}
Each $d$-dimensional shadow of $\mathcal{R}(N, R_1, \hdots, R_{\ell + 1})$ is contained in the shadow of the slice of $\mathcal{R}(N, R_1, \hdots, R_{\ell + 1})$ where the coordinates perpendicular to the shadow are equal to 1. Also when any $a_j=1$ we obtain constant bounds for $a_{p-j}$ from above. This will allow us to ensure that the region defined below depends on an even number of variables: if the dimension of the shadow is odd then, without loss of generality, we know that at least one variable, $a_j$, appears without $a_{p-j}$. As we only care about the hypervolume of the shadows up to a constant, the constant bound on $a_j$ can then be used to exclude this from the region we define above. Doing this for all such $a_j$ we are left to worry only about the pairs $(a_i, a_{p-i})$ in the shadow. 

\end{remark}

\begin{example}
	When $p=5$ we have $r_1 = \frac{a_1}{a_4}$ and $r_2 = \frac{a_2}{a_3}$ and, without loss of generality, we have something of the form:
		\[	R_1 \leq r_1r_2^3 < r_1^3 r_2^{-1} \leq R_3.	\]
	In this case we have that 
		\[	R_1^3 < r_1^9r_2^{-3} \leq R_3^3	\]
	which we can multiply with the inequality above to gives us that:
		\[	R_1^4 \leq r_1^{10} < R_3^4	\]
	or $R_1^{2/5} \leq r_1 < R_3^{2/5}$. Isolating the other ratio we see that:
		\begin{align}
			R_1^3R_3^{-1} 						&	\leq r_2^{10} \leq R_1^{-1}R_3^3		\\
			\left(\frac{R_1^3}{R_3}\right)^{1/10} 		&	\leq r_2 \leq \left(\frac{R_3^3}{R_1}\right)^{1/10}. \label{5r2bound}
		\end{align}
	Now, consider the shadow given by $a_1a_2a_3$. Following the strategy from above we know that $a_1$ is bounded by constants and so to bound this shadow we only need to consider the region defined by:
		\[	\mathcal{R}_2(N) = \left\{ (a_2,a_3)\in \RR^2\mid 1< a_2a_3< N, b_1 < \frac{a_2}{a_3}< B_1\right\}	\]
	where $b_1$ and $B_1$ are the bounds given in $\ref{5r2bound}$. Using the change of variables:
		\begin{align*}	x_1 = \frac{a_2}{a_3} 	&&	x_2 = \frac{a_2a_3}{x_1} = a_3^2	\end{align*}
	we have the Jacobian matrix
		\[	\begin{pmatrix}	
				\frac{1}{a_3}	&	\frac{-a_2}{a_3^2}	\\
				0			&	2a_3	
			\end{pmatrix}
		\]
	which has determinant $2$ and we obtain the following bounds on these new variables:
		\begin{align*}	1 < x_1x_2 < N 	&&		b_1<x_1 <B_1.		\end{align*}
	Now, the volume of the shadow is approximated by:
		\begin{align*}
			\Vol(R_2(N)) 	&	= \int_{b_1}^{B_1} \int_0^{N/x_1} \frac{1}{2} dx_1 dx_2	 \\
						&	= \int_{b_1}^{B_1} \frac{N}{2x_1}dx_1 				\\
						&	= \frac{N\log\left(\frac{B_1}{b_1}\right)}{2}				\\
						&	= \mathcal{O}(N)
		\end{align*}
	The same can be done with any other shadows which shows that the maximum volume of any lower-dimensional shadow of $\mathcal{R}(N, R_1, R_2, R_3)$ is bounded by $\mathcal{O}(N)$.
\end{example}

Let $d$ be even, $i\in\{1, \hdots, \frac{d-1}{2}\}$, and define the region
	\[	\mathcal{R}_{d}(N, \{b_i, B_i\}) := \left\{(a_1, \hdots, a_{d})\in \RR^{d} \mid 1<a_1\cdots a_d<N, b_{i}<\frac{a_i}{a_{d-i}}<B_{i}\right\}.	\]
We will show that it is sufficient to consider the volume of such a region when computing the volume of the lower-dimensional shadows. First, we show that:
\begin{lemma}\label{shadowvol}
	The hypervolume of the region $\mathcal{R}_{d}(N)$ is $\mathcal{O}(N\log^{(d-3)/2}(N))$.
\end{lemma}

\begin{proof}
	The proof will follow the same strategy as \cref{volumeprop}. We define the following change of variables:
		\begin{align*}
			x_1 		& 	= \frac{a_1}{a_{d}}		\\
			x_2 		& 	= a_1				\\
			x_3		&	= \frac{a_2}{a_{d-1}} 		\\
			x_4		&	= a_2				\\
					&	\vdots				\\
			x_{d-1}	&	= \frac{a_{d/2}}{a_{d/2+1}}	\\
			x_d		&	= \frac{a_1\cdot a_2\cdots a_d}{x_1\cdot x_2\cdots x_{d-1}}			
		\end{align*}
	The Jacobian determinant of this change of variables is constant and equal to $\pm 2$, this follows \cref{shadowconstant}\footnote{The Jacobian here is obtained from that in \cref{shadowconstant} by permuting columns}. Then, for $i\in \{1, 3, \hdots, d-1\}$ we have the following bounds: $b_{(i+1)/2}<x_i<B_{(i+1)/2}$. We obtain bounds on the remaining variables using the fact that $1< x_1\cdots x_d<N$ and this gives us:
		\begin{align*}
		 	\Vol(\mathcal{R}_d(N, \{b_i, B_i\}))	&	=	\int_{b_1}^{B_1}\cdots\int_{b_{d-1}}^{B_{d-1}}\int_1^N\cdots \int_0^{N/\prod_{i=1}^{d-1} x_i} \frac{1}{2} dx_d\cdots dx_1 	\\
								&	=	\frac{1}{2}\int_{b_1}^{B_1}\cdots\int_{b_{d-1}}^{B_{d-1}}\int_1^N\cdots \int_{1/\prod x_i}^{N/\prod x_i} \frac{N}{x_1\cdots x_{d-1}} dx_{d-1} \cdots dx_{1}	\\
								&	=	\frac{N\log^{d/2 - 1}(N)}{2}\int_{b_1}^{B_1}\cdots \int_{b_{d-1}}^{B_{d-1}} \frac{1}{x_1x_2\cdots x_{d/2}} dx_1 dx_2\cdots dx_{d/2}					\\
								&	=	\frac{N\log^{d/2 - 1}(N)}{2}\prod_{i=1}^{d/2} \log\left(\frac{B_i}{b_i}\right)		\\
								& 	= 	\mathcal{O}(N\log^{d/2 - 1}(N)).\qedhere
		\end{align*}
\end{proof}

This gives the hypervolume of the region $\mathcal{R}_d(N, \{b_i, B_i\})$ and we use this to prove \cref{shadowprop}.

\begin{proof}[Proof of \cref{shadowprop}]
For the region $\mathcal{R}(N, R_1, \hdots, R_{\ell+1})$, take any $d$ dimensional shadow. If $a_j$ is in the shadow but $a_{p-j}$ is not then we have constant bounds on $a_j$ and we can drop the dimension of the shadow in consideration by 1: we do this for all $a_j$ until the only variables in the shadow are the pairs $(a_j, a_{p-j})$. We note that this will only change the error by a constant which will not affect our result. Once this is complete we are considering a region $\mathcal{R}_{d'}(N)$ where $d'$ is even and we have:
	\[	\Vol(\mathcal{R}_{d}) = \mathcal{O}(\Vol(\mathcal{R}_{d'})) = \mathcal{O}(N\log^{d'/2 -1}(N)).	\]
Now, this approximation will increase--by $\log(N)$--each time we add a pair of variables$(a_j, a_{p-j})$ to the shadow. As the largest, even, $d$ that we could consider is $d=p-3$ this shows that the maximum volume of any lower dimensional shadow is:
	\[	\mathcal{O}\left(N\log(N)^{\frac{p-3}{2} - 1}\right) = \mathcal{O}(N\log(N)^{\ell - 2}).	\]

\end{proof}

Combining propositions \ref{volumeprop} and \ref{shadowprop} we have the following:

\begin{theorem}\label{volume}
	For any prime $p$, the volume of $\mathcal{R}(N, R_1, R_2, \hdots, R_{\ell+1})$ is 
		\[	c(N-1)\log^{\ell-1}(N)\cdot H(R_1, \hdots, R_{\ell+1})	\]
	where $H(R_1, \hdots, R_{\ell+1})$ is the homogeneous polynomial in \cref{homogeneous}. 
	Moreover, the maximum measure of this region's lower-dimensional shadows is $O\left(N\log^{\ell-2}(N)\right)$. 
\end{theorem}

\subsection{Adventure into the land of Jacobian determinants}

In this section we prove results about the Jacobian determinants used in the proofs of \cref{volumeprop} and \cref{shadowvol}.

\subsubsection{Jacobian determinant in \cref{volumeprop}}
Here we prove that the Jacobian determinant, \ref{jacobiandet}, is non-zero. We define the change of variables again:
	\[	x_i	 = 	\begin{cases}	
					\lambda_i^p	&	 \text{ for } i\in \{1, \hdots, \ell\}		\\	
					a_{i-\ell}\cdot a_{p-i+\ell}	&	 \text{ for } i\in \{\ell+1, \hdots, p-2\}	\\ 
				\end{cases}	\]
where $\lambda_i$ are given in \cref{shapeparameters}.

\begin{proposition}\label{jacobiandeterminant}
	The Jacobian determinant of this change of variables from $(a_1, a_2, \hdots, a_{p-1})$ to $(x_1, x_2, \hdots, x_{p-1})$ is constant and equal to:
		\[	\pm 2^{p-2} p^{\ell-1}h_p^-.		\]
\end{proposition}

We will prove this proposition by showing that the determinant is equal to 
	\[	\pm 2^{\ell}\cdot \det\left(\left(2ij - p - 2p\left\lfloor \frac{ij}{p}\right\rfloor\right)_{1\leq i,j \leq \ell}\right)	\]
and then appeal to a result of Carlitz and Olson, \cite{carlitzolson}, regarding Maillet’s determinant. Before this we provide a short example, again with $p=5$:

\begin{example}
	We saw, in \cref{shapeparametersex5},  that the shape parameters of $K=\QQ(\sqrt[5]{m})$ are:
	\begin{align*}
		\lambda_1 = \left(\frac{a_4^3a_3}{a_1^3a_2}\right)^{1/5}	&&	\lambda_2 = \left(\frac{a_4a_2^3}{a_1a_3^3}\right)^{1/5}
	\end{align*}
	We will, without loss of generality, assume that $\lambda_1<\lambda_2$ and so $x_1 = \lambda_1^p$ and $x_2= \lambda_2^p$. Then $x_3 = a_1\cdot a_4$ and
	\[	x_4 = \frac{ a_1a_2a_3a_4}{x_1x_2x_3} = \frac{ a_1^4 a_3^3}{a_4^4 a_2}	\]
	The Jacobian of this change of variables is given by:
	\[ 	\Jac\left(\frac{\partial x_i}{\partial a_j}\right) = 	\begin{pmatrix}
												\frac{-3a_4^3a_3}{a_1^4a_2}	&	\frac{-a_4^3a_3}{a_1^3a_2^2}	&	\frac{a_4^3}{a_1^3a_2}		&	\frac{3a_4^2a_3}{a_1^3a_2}	\\
												\frac{-a_4a_2^3}{a_1^2a_3^3}	&	\frac{3a_4a_2^2}{a_1a_3^3}	&	\frac{-3a_4a_2^3}{a_1a_3^4}	&	\frac{a_2^3}{a_1a_3^3}	\\
												a_4	&	0	&	0	&	a_1	\\
												\frac{4a_1^3 a_3^3}{a_4^4 a_2}	&	\frac{-a_1^4 a_3^3}{a_4^4 a_2^2}	&	\frac{3a_1^4 a_3^2}{a_4^4 a_2}	&	\frac{ -4a_1^4 a_3^3}{a_4^5 a_2}	
											\end{pmatrix}	\]
	And, finally, the determinant is $\left| \Jac\left(\frac{\partial x_i}{\partial a_j}\right)\right| = 40 =  2^3\cdot 5$. 
	\\
	
	For sake of the arguments below, we define the coefficient matrix, $C_p$, of $\Jac\left(\frac{\partial x_i}{\partial a_j}\right)$ be the matrix of coefficients:
		\[	C_5 = 	\begin{pmatrix}	
						-3	&	-1	&	1	&	3	\\
						-1	&	3	&	-3	&	1	\\
						1	&	0	&	0	&	1	\\
						4	&	-1	&	3	&	-4	\\
					\end{pmatrix}.
		\]
	The sum of the entries in any column is equal to $1$ and so we can replace the bottom row with the sum of all rows without changing the determinant. Doing this we obtain
		\[	C_5'	= 	\begin{pmatrix}	
						-3	&	-1	&	1	&	3	\\
						-1	&	3	&	-3	&	1	\\
						1	&	0	&	0	&	1	\\
						1	&	1	&	1	&	1	\\
					\end{pmatrix}.
		\]	
	Finally, from columns 3 (resp. 4) we subtract columns 2 (resp. 1) to obtain:
		\[	C_5''	= 	\begin{pmatrix}	
						-3	&	-1	&	2	&	6	\\
						-1	&	3	&	-6	&	2	\\
						1	&	0	&	0	&	0	\\
						1	&	1	&	0	&	0	\\
					\end{pmatrix}.
		\]	
	Then we have that:
		\begin{align*}
			 \det(C_5) 	&	= \det(C_5') 		\\
			 			&	= \det(C_5'')		\\
						&	= 1 \cdot 1 \cdot\det\begin{pmatrix}	2	&	6	\\	-6	&	2	\end{pmatrix}	\\
						&	=2^2 \det\begin{pmatrix} 1	&	3	\\	-3	&	1\end{pmatrix}					\\
						&	=	2^3\cdot 5	\\
						&	=	\det\left(\Jac\left(\frac{\partial x_i}{\partial a_j}\right)\right).	
		\end{align*}
	This argument will pass through in general to show that the Jacobian determinant is equal to 
		\begin{align}\label{jacobianformulap}
			\det\left(\Jac\left(\frac{\partial x_i}{\partial a_j}\right)\right) & = 2^{\ell}\cdot \det\left(\left(2ij - p - 2p\left\lfloor \frac{ij}{p}\right\rfloor\right)_{1\leq i,j \leq \ell}\right).
		\end{align}
	
\end{example}

\begin{remark}
We believe that when $p$ is replaced by any odd integer $n$, and one defines the same Jacobian change of variables, that the resulting determinant should be non-zero. One difference is that the statement in \ref{jacobianformulap} is not true when $n$ is odd and composite. In this case the matrix $C_n'$ is not anti-symmetric about the columns and instead the formula for the determinant of the Jacobian becomes:
	\begin{align}\label{jacobianformulad}
		\det\left(\Jac\left(\frac{\partial x_i}{\partial a_j}\right)\right) & = 2^\ell \cdot \det\left(\left(2ij - nj - n\left\lfloor \frac{ij}{n}\right\rfloor + n\left\lfloor \frac{(n-i)j}{n}\right\rfloor\right) \right)_{1\leq i,j \leq \ell}.
	\end{align}
We utilize a result that holds for primes but are unaware of such a statement when $p$ is replaced by an odd composite number. 

\end{remark}

\begin{proof}[Proof of \cref{jacobiandeterminant}]
Note that the change of variables depends on the ordering of the $\lambda_i^p$. In general we cannot assume any particular ordering but this does not change the absolute value of the determinant as switching $\lambda_i$ and $\lambda_j$ only changes the determinant by a factor of $-1$. Similarly, switching $x_i = \lambda_i$ with $x_i=\lambda_i^{-1}$ multiplies the determinant by $-1$. As such, we need only prove this for a particular choice of $\{\lambda_i\}$.

\noindent
We first show that the determinant is a constant, which will allow us to focus on a more simple matrix, namely the coefficient matrix $\mathcal{C}_{p}$ which we define below. 

We have that the Jacobian determinant of the change of variables is given by:
	\[	\sum_{\sigma\in S_{p-1}} (\sgn(\sigma))\cdot \prod_{i=1}^{p-1} J_{i, \sigma(i)}	\]
where the product is over elements of the form $J_{i, \sigma(i)} = \frac{\partial x_i}{\partial a_{\sigma(i)}}$. But we also have that $x_i$ is a ratio of the $a_j$ as given in \ref{shapeparameters} and hence we have the following relationship:
	\[	J_{i, \sigma(i)} = \frac{\partial x_i}{\partial a_{\sigma(i)}} = c_{\sigma(i)}\frac{ x_i}{a_{\sigma(i)}}	\]
where $c_{\sigma(i)}$ is a constant, equal to the exponent of $a_{\sigma(i)}$ in $x_j$. Looking again at the product above we have:
	\begin{align*}
		\prod_{i=1}^{p-1} J_{i, \sigma(i)} 	&	= \prod_{i=1}^{p-1} \frac{\partial x_i}{\partial a_{\sigma(i)}} \\
									&	= \prod_{i=1}^{p-1} \frac{c_{\sigma(i)} x_i}{a_{\sigma(i)}} 	\\
									&	= \prod_{i=1}^{p-1} c_{\sigma(i)} \frac{a_i}{a_{\sigma(i)}} 	\\
									&	= \prod_{i=1}^{p-1} c_{\sigma(i)}					\\
									&	= c_\sigma.		
	\end{align*}
	This shows that each term in the sum above is a constant, and therefore that the determinant is a constant:
	\begin{align*}
		\det(J)	&	= \sum_{\sigma \in S_{p-1}} \left(\sgn(\sigma)\cdot c_{\sigma} \right)
	\end{align*}
\noindent
	so we can forget about the variables in the Jacobian and focus instead on the matrix of coefficients, $C_p$. This matrix has the property that the sum of any column is 1, and we can therefore replace the bottom row with the row of all 1s. We call this matrix $C_p'$ and note the formula:
		\[	
			C_p'(i,j) = \begin{cases}
						\left(2ij - p - 2p\left\lfloor \frac{ij}{p}\right\rfloor \right)	&	\text{ if } 1\leq i \leq \ell \text{ and } 1\leq j \leq p-1				\\
						0										&	\text{ if } \ell < i \leq p-2 \text{ and } j \neq \ell - i \text{ or } j \neq p-\ell-i	\\
						1										&	\text{ otherwise.}
					\end{cases}	
		\]	
	Now the matrix $C_p'$ is anti-symmetric about the columns for $1\leq i \leq \ell$, so $C_p'(i,j) = - C_p'(i, p-j)$, and symmetric for $\ell < i \leq p-1$. Subtracting the $i^{th}$ column from the $p-i^{th}$ column yields a matrix:
		\[	C_p'' = 	\left(
						\begin{array}{c | c} 
							C_p'(i,j)_{1\leq i,j\leq \ell}						&	2\cdot C_p'(i,j)_{1\leq i,j\leq \ell}		\\\hline
							C_p'(i,j)_{\substack{\ell< i\leq p-1 \\ 1\leq j \leq \ell}}	&	0	
						\end{array}
					\right).
		\]
	The matrix $C_p'(i,j)_{\substack{\ell< i\leq p-1 \\ 1\leq j \leq \ell}}$ is lower triangular with determinant 1 which shows that 
		\begin{align*}
			\det(C_p) = \det(C_p'') 	&	= \det(2\cdot C_p'(i,j)_{1\leq i,j\leq \ell})	\\
								&	= 2^\ell \cdot \det\left(\left(2ij - p - 2p\left\lfloor \frac{ij}{p}\right\rfloor\right)\right)_{1\leq i,j \leq \ell}
		\end{align*}

What remains is to show that $\det\left(\left(2ij - p - 2p\left\lfloor \frac{ij}{p}\right\rfloor\right)\right)_{1\leq i,j \leq \ell}$ is non-zero. In \cite{carlitzolson} the authors show that Maillet's determinant is non-zero and equal to $\pm p^{(p-3)/2}h_p^-$. Note that Maillet's determinant, which we denote $D_p$, is the determinant of the matrix $(R(rs'))$ where $R(r)$ is the least positive residue of $r$, and $s'$ is such that $ss' \equiv 1 \mod{p}$. One step in proving this involves showing that the determinant:
	\[	D_p' = \det\left(R(rs') - \frac{p}{2}\right)	= \frac{-1}{2}D_p.	\]
However, since $R(rs')- \frac{p}{2}  =  rs' - \frac{1}{2} - p\left\lfloor \frac{rs'}{p}\right\rfloor$, we only need to multiply this matrix $D_p'$ by 2 to obtain the matrix above (and apply some number of column swaps which will only change the determinant by a sign). This shows that the determinant of the change of variable matrix is equal to:
	\begin{align*}
												\det(J) 							&	= \pm2^\ell \cdot \det\left(\left(2ij - p - 2p\left\lfloor \frac{ij}{p}\right\rfloor\right)\right)_{1\leq i,j \leq \ell}	\\
																				&	= \pm2^\ell \cdot 2^\ell \cdot D_p'	\\
																				&	= \pm2^{p-1} \cdot  \frac{1}{2} D_p		\\
																				&	= \pm2^{p-2} p^{\ell-1} h_p^-
	\end{align*}
as desired.

\end{proof}

\subsubsection{Jacobian determinant in \cref{shadowvol}}\label{shadowsection}

Let $d$ be an even integer and define the following change of variables:
		\[	x_i = \begin{cases} \frac{a_i}{a_{i+1}} 			&	\text{ if } i \text{ is odd }				\\
							a_{i-1}					&	\text{ if } i \text{ is even and } i \not=d		\\
							\frac{N}{\prod_{j=1}^{d-1}x_i}	&	\text{ if } i=d						\\
				\end{cases}.
		\]
We note that this differs a bit from the change of variables seen in \cref{shadowvol} but, as was mentioned in its proof, we can go between these by simply permuting the columns of the matrix we obtain. 
\begin{lemma}\label{shadowconstant}
	The Jacobian determinant of this change of variables from $(a_1, a_2, \hdots, a_d)$ to $(x_1, x_2, \hdots, x_d)$ is constant and equal to $2$. 
\end{lemma}

\begin{proof}
	We stated above that the determinant is constant, by the same argument as seen in the proof of Lemma \cref{jacobiandeterminant}, and therefore that we only need to worry about the coefficients in the Jacobian matrix. The coefficient matrix can be defined recursively as follows:
		\[	C_2 = \begin{pmatrix} 	1 	&	-1	\\
								0	&	2	
				\end{pmatrix}	
		\]
	and 	
		\[	
			C_d  = \left(\begin{array}{c c | c c c } 	
					1 		&	-1	&	0	&	\cdots 	&	0	\\
					1		&	0	&	0	&	\cdots 	&	0	\\\hline
					0		&	0	&		&			&		\\
					\vdots 	& \vdots	&		&	C_{d-2}	&		\\	
					0		&	0	&		&			&		\\
					-1		&	2	&		&			&		\\
					
				\end{array}	
			\right)
		\]
		
Since the determinant of $C_2=2$, and $\det(C_d) = \det\begin{pmatrix} 1 & -1 \\ 1 & 0 \end{pmatrix} \cdot \det(C_{d-2}) = \det(C_{d-2})$ we have that $\det{C_d} = 2$ for all $d$. 
\end{proof}

\subsection{Principle of Lipschitz}
To count integer points in the region, $\mathcal{R}(N, \{R_i\})$ we define its intersection with $\mathcal{L} = \ZZ^{p-1}$ as
	\[	\mathcal{R}_\mathcal{L}(N, \{R_i\})	:= R(N, R_1,\hdots, R_{\ell+1}) \cap \mathcal{L}\]
and using the principle of Lipschitz we obtain the following corollary:

\begin{corollary}
	For $p$, $N, R_i\in \RR$, and $R_1<R_2<\hdots<R_{\ell+1}$  as above we have:
	\[	\#\mathcal{R}_\mathcal{L}(N, R_1, \hdots, R_{\ell+1}) = \frac{N\log(N)^{\ell-1}}{c_p\cdot \ell!} H(R_1, \hdots, R_\ell) + O\left( N\log(N)^{\ell-2}\right)	\]
\end{corollary}

What we really need to count are the integer points in this region which satisfy the infinitely many congruence conditions coming from our strongly carefree condition.

\subsubsection{Point Count/Sieve}

As stated above, we have translated our problem of counting number fields of bounded discriminant and shape restrictions to counting certain strongly carefree tuples in some specific regions.  
Since we have already computed the volume of the region, and estimated the number of lattice points in said region, we must now apply a strongly carefree sieve. 

For this section we will try to avoid using $p$ entirely, and will consider the general case of $n$-tuples. The strongly carefree condition will impose infinitely many congruence conditions and will correspond to the tuples $(a_1, \hdots, a_n)$ such that (for all primes $q$) $a_i\not\equiv 0 \mod{q^2}$ and if $a_i\equiv 0 \mod{q}$ then there is no other $a_j\equiv 0 \mod{q}$ for $j\not=i$. These are the squarefree, and relatively prime conditions respectively and this set will be denoted, as above, by $\mathcal{SC}_n$. 

We begin by defining the following set: 
	\[	\mathcal{C}_{n,q}^{cf}	:= \left\{(a_1,\hdots, a_n)\in \left(\ZZ/q^2\ZZ\right)_{\not=0}^n: \text{ at least two of the } a_i \text{ are } 0\mod{q} \right\}	\]

This, being the complement of the set that we desire, will play an important role in the application of the elementary sieve and has a particularly nice expression (for a fixed prime $q$):
	\begin{align}	\# \mathcal{C}_{n,q}^{cf} = \sum_{i=2}^n \binom{n}{i} (q-1)^i (q^2-q)^{n-i}	\end{align}
This is obtained by simply choosing $i$ entries to be $0\mod{q}$; noting that for each of the $i$ entries there are $(q-1)$ such values. The other entries only need to avoid multiples of $q$ and whence there are $(q(q-1))$ possibilities for the $n-i$ remaining entries\footnote{note we have avoided any entries that are $0 \mod{q^2}$}. Letting $\mathcal{C}_{n,q}^{s}:= \{(a_1,a_2,\hdots, a_n)\in \ZZ^n: a_i\equiv 0 \mod{q^2} \text{ for some } i\in \{1,\hdots, n\}\}$, we define the set:
	\[	\mathcal{C}_{n,q}^{scf} = \mathcal{C}_{n,q}^{s} \cup \mathcal{C}_{n,q}^{cf}	\]
noting that this will gather all tuples (in $\ZZ/q^2\ZZ$) which are not squarefree, and not relatively prime with respect to $q$. We denote the complement of this set in $\ZZ/q^2\ZZ$ by	$\mathcal{SC}_{n,q}$ and, finally, we note that:
	\begin{align}	\#\mathcal{SC}_{n,q} = (q-1)^n(q^n+nq^{n-1})	\end{align}

\subsubsection{$q$-adic density} 
For all $Y\geq 2$ we define the following:
	\[	n(Y) := \prod_{q\leq Y} q^2	\]
and $\mathcal{C}_{n,Y}$ denote the set of congruence conditions, $\mathcal{C}_{q,n}^{scf}$, modulo $n(Y)$ for $1<q\leq Y$. 
Letting
	\[	\mathcal{L}_n(Y) := \{(a_1,a_2,\hdots, a_n)\in \ZZ^n: (a_1, a_2, \hdots, a_n)\not\in \mathcal{C}_{n,Y}	\}	\]
we obtain the following $q$-adic density of $\mathcal{L}_n(Y)$:
\begin{align*}
	\delta_q(\mathcal{L}_n(Y))		&	=  \frac{\#\mathcal{SC}_{n,q}}{q^{2n}}	\\	
							&	= \frac{(q-1)^n(q+n)q^{n-1}}{q^{2n}} 		\\
							&	= \frac{(q-1)^n(q+n)}{q^{n+1}}			\\
							&	= \frac{q^{n+1} - n^2 q^{n-1} + (q+n)\left(\sum_{i=2}^n\binom{n}{i}q^{n-i}(-1)^i\right)}{q^{n+1}}	\\
							&	= 1 - \frac{n^2}{q^2} + (q+n)\sum_{i=2}^n\binom{n}{i}q^{-i-1}(-1)^i	
\end{align*}

\subsubsection{Application to pure prime degree fields}
Using the notation above and a standard argument, as in \cite{galoisquartic}, we have the following corollary which will lead us to the main result of this paper. 
\begin{corollary}
	For $p$ prime and $1<R_1<R_2<\hdots<R_{\ell+1}$
	\[	\mathcal{R}_{\mathcal{L}_{p-1}(Y)}(N, \{R_i\}) = c_p^{-1} \prod_{q\leq Y} \delta_q(\mathcal{L}_{p-1}(Y)) N\log(N)^{\ell-1}H(R_1, \hdots, R_{\ell+1}) + O\left(N\log(N)^{\ell -2}\right)	\]
\end{corollary}
What remains is to show that we may take the limit of $Y$. This uses one of the basic sieve techniques which appears in many related works, see \cite{davenportheilbron, galoisquartic, multiquad}, and whose notation we try to emulate here. Since $p$ is a fixed prime the length of the $n$-tuple is now $p-1$ so we exclude this from $\mathcal{L}_n(Y)$ and simply write $\mathcal{L}(Y)$. Letting $\mathcal{L}_{\infty}$ be the set with congruence conditions applied at all primes, we have a naive upper bound
\begin{align*}
	\limsup_{N\rightarrow \infty} \frac{\#\mathcal{R}_{\mathcal{L}_\infty}(N, R_1, \hdots, R_{\ell+1})}{N\log(N)^{\ell-1}}	
		&	\leq \lim_{Y\rightarrow \infty}\lim_{N\rightarrow \infty}\frac{\#\mathcal{R}_{\mathcal{L}(Y)}(N, R_1, \hdots, R_{\ell+1})}{N\log(N)^{\ell-1}}	\\
		&	\leq c_p^? \prod_{q\leq Y} \delta_q(\mathcal{L}(Y)) H(R_1, \hdots, R_{\ell+1})		
\end{align*}
We now set $\mathcal{W}_q := \{(a_1, \hdots, a_{p-1})\in \ZZ^{p-1}: (a_1,\hdots, a_{p-1}) \in \mathcal{C}_q^{scf}\}$, and write:
	\[	\mathcal{R}_{\mathcal{L}(Y)}(N, R_1, \hdots, R_{\ell+1}) \subseteq \mathcal{R}_{\mathcal{L}_\infty}(N, R_1, \hdots, R_{\ell+1}) \cup \bigcup_{q>Y} \mathcal{R}_{\mathcal{W}_q}(N, R_1, \hdots, R_{\ell+1})		\]	
This allows us to bound the quantity below as follows:
	\begin{align}\label{bigo}	
	\frac{\#\mathcal{R}_{\mathcal{L}_\infty}(N, R_1, \hdots, R_{\ell+1})}{N\log(N)^{\ell-1}} \geq \frac{\mathcal{R}_{\mathcal{L}(Y)}(N, R_1, \hdots, R_{\ell+1})}{N\log(N)^{\ell-1}} - O\left( \sum_{q>Y} \frac{\mathcal{R}_{\mathcal{W}_q}(N, R_1, \hdots, R_{\ell+1})}{N\log(N)^{\ell-1}}\right).	
	\end{align}
By the previous section, we know that
	\[	 \frac{\mathcal{R}_{\mathcal{W}_q}(N, R_1, \hdots, R_{\ell+1})}{N\log(N)^{\ell-1}}	= O(q^{-2})	\]
which is enough to show that in the limit, as $Y\rightarrow \infty$, the big O term on the RHS of \ref{bigo} tends to 0. Taking the $\liminf$ of \ref{bigo}, in conjunction with the $\limsup$ above, we have the desired result:

\[	\mathcal{R}_{\mathcal{L}_\infty}(N, R_1, \hdots, R_{\ell+1}) = c_p^{-1} \prod_{q} \delta_q(\mathcal{L}_\infty) N\log(N)^{\ell-1}H(R_1, \hdots, R_{\ell+1}) + o\left(N\log(N)^{\ell -1}\right)	\]

\noindent
In what follows we write $\delta_q$ for the $q$-adic density $\delta_q(\mathcal{L}_\infty)$.

The only thing remaining is to apply this to the case at hand: namely, to count the number of fields whose discriminant is bounded by $N$ and whose shape lies in $\{R_i\}$. The total number of such fields is given by:
	\begin{align*}
			\mathcal{N}_p(N, \{R_i\}) 	&	= c_p^{-1}  \prod_{q} \delta_q N\log(N)^{\ell-1}H(R_1, \hdots, R_{\ell+1}) + o\left(N\log(N)^{\ell -1}\right)		\\
								&	= (\pm 2^{p-2}p^{\ell-1}h_p^-)^{-1} \prod_{q} \delta_q N\log(N)^{\ell-1}H(R_1, \hdots, R_{\ell+1}) + o\left(N\log(N)^{\ell -1}\right).	
	\end{align*}

For wild (type I) fields we have that 
	\[	\Delta(K) = -p^p\prod_{i=1}^{p-1} a_i^{p-1} 	\implies N = \frac{1}{p}\cdot \left(\frac{X}{p}\right)^{1/(p-1)}\]	
and for tame (type II) fields we have 
	\[	\Delta(K) = -p^{p-2}\prod_{i=1}^{p-1} a_i^{p-1} 	\implies N = \left(\frac{X}{p^{p-2}}\right)^{1/(p-1)}.\]
	
Using this we can determine the asymptotics for the number of pure prime degree fields of type I (resp. type II) with absolute discriminant bounded by $X$ and shape in $\{R_i\}$:	
	\begin{align*}
		\mathcal{N}_{p}^I(X, \{R_i\})	= &		\frac{2p-2}{(2p-1)2^{p-2}p^{\ell+\frac{1}{p-1}}h_p^-} \prod_q \delta_q X^{1/(p-1)}\log(X)^{\ell-1}H(\{R_i\}) + o(X^{1/(p-1)}\log(X)^{\ell-1})			\\
		\mathcal{N}_{p}^{II}(X, \{R_i\})	= &		\frac{1}{(2p-1)2^{p-2}p^{\ell-1+\frac{p-2}{p-1}}h_p^-} \prod_q \delta_q X^{1/(p-1)}\log(X)^{\ell-1}H(\{R_i\}) + o(X^{1/(p-1)}\log(X)^{\ell-1})  
	\end{align*}

This, together with \cref{wildmeasure}, proves Theorem \ref{equidistribution}: the (regularized) equidistribution of shapes in the family of pure prime degree number fields.

\subsection{Examples} We end this section with two examples: first we recover the result of Harron, in \cite{purecubics}, and second we state the result in the case of pure quintic fields. 

\begin{example}
	When $p=3$ the negative part of the class number is 1 and so we have:
	\begin{align*}
		\mathcal{N}_3^I(X, \{R_1, R_2\})	&	 = 	\frac{4}{5\cdot 2 \cdot 3\cdot 3^{\frac{1}{2}}} \prod_q \delta_q X^{1/2} H(R_1, R_2) + o(X^{1/2})	\\
									&	=	\frac{2}{15\sqrt{3}} \prod_q \left(1 - \frac{3}{q^2} + \frac{2}{q^3}\right) X^{1/2} \log\left(\frac{R_2}{R_1}\right) + o(X^{1/2})	\\
		\mathcal{N}_3^{II}(X, \{R_1, R_2\})	&	=	\frac{1}{5\cdot 2 \cdot 3^{\frac{1}{2}}} \prod_q \delta_q X^{1/2} H(R_1, R_2) + o(X^{1/2})	\\
									&	=	\frac{1}{10\sqrt{3}} \prod_q \left(1-\frac{3}{q^2} + \frac{2}{q^3}\right) X^{1/2} H(R_1, R_2) + o(X^{1/2})
	\end{align*}
\end{example}

\begin{example} 
For $p=5$ the negative part of the class number is again 1, we have that $\delta_q =  \left(1- \frac{10}{p^2}+\frac{20}{p^3} - \frac{15}{p^4} + \frac{4}{p^5}\right)$ and so
	\[	\mathcal{N}_5^I(X, \{R_1, R_2, R_3\})=	\frac{1}{225\sqrt[4]{5}} \prod_q \delta_q X^{1/4}\log(X)\cdot \left(\log^2\left(\frac{R_3}{R_1}\right) - \log^2\left(\frac{R_2}{R_1}\right)\right)+ o(X^{1/4}\log(X))	\]
and 
	\[	\mathcal{N}_5^{II}(X, \{R_1, R_2, R_3\}) =	\frac{1}{360\sqrt[4]{5^3}} \prod_q \delta_q X^{1/4}\log(X)\cdot \left(\log^2\left(\frac{R_3}{R_1}\right) - \log^2\left(\frac{R_2}{R_1}\right)\right)+ o(X^{1/4}\log(X))	\]
\end{example}

\section{Pure and $F_p$ fields}\label{frobeniussection}

In this last section we prove our final theorem, that the pure prime degree number fields are exactly those $F_p$ number fields, with degree $(p-1)$-resolvent equal to $\QQ(\zeta_p)$. We begin by discussing the group theoretic prerequisites, defining the group $F_p$, and defining what exactly we mean by the resolvent field. Once we have the relevant tools and definitions in hand we prove the aforementioned claim which allows us to phrase the above study in the flavor of Malle.

\subsection{Group Theory}  We first define Frobenius groups in general (following the nice narrative laid out by Terrance Tao, in \cite{taoblog}), and state some well known facts about these groups. We will then define the specific Frobenius group, $F_p$, to which we refer and state some facts that will be of use for us below. 
\begin{definition} 
	We call $G$ a Frobenius group if there exists a subgroup $H$ of $G$ such that: $\left(H\cap gHg^{-1}\right) = \{1\}$ for all $g\in G\setminus H$. This yields the following decomposition:
	\[	G = \bigcup_{gH\in G\setminus H} (gHg^{-1}\setminus\{1\}) \cup J	\]	
\end{definition}

We note that $H$ and $J$ are often referred to as the {\bf Frobenius complement} and {\bf Frobenius kernel}, respectively. The following theorem will then provide the motivation for our definition of $F_p$ that we use throughout:

\begin{theorem}[Frobenius (1901)]
	Let $G$ be a Frobenius group with complement $H$ and kernel $J$. Then $K\triangleleft G$ and $G = J \rtimes H$. 
\end{theorem}

This theorem shows us that, in general, Frobenius groups can be realized as a semi direct product of the complement and the kernel and hence we define the Frobenius group, $F_p$, as follows:
	\[	F_p := \mathbb{F}_p \rtimes \mathbb{F}_{p}^\times = \langle \sigma, \tau \mid \sigma^p = \tau^{p-1} = 1, \tau\sigma\tau^{-1} = \sigma^g\rangle	\]
where $g$ is a primitive root modulo $p$. As $C_p$ is a normal Sylow $p$ subgroup of $F_p$, we define the \textbf{degree $(p-1)$-resolvent} of $\tilde{K}$, which we will refer to simply as the resolvent (the degree being clear), to be the unique fixed field of $C_p$ in $\tilde{K}$. We denote the resolvent of $\tilde{K}$ by $K_{p-1}$ and we have the following field diagram:
\[
\begin{tikzcd}
												&	\tilde{K}		\\
	K_{p-1}\arrow[dash]{ur}{\langle \sigma \rangle}			&	K\arrow[dash]{u}		\\
	\QQ\arrow[dash]{u}{\langle \tau \rangle}\arrow[dash]{ur}
\end{tikzcd}
\]

\subsection{Pure prime degree fields and $F_p$ fields}
With the necessary background covered we now turn our attention to the proof of Theorem \ref{theoremgalois}.

\begin{proof}
($\Rightarrow$) If $K$ is pure then $K\cong \QQ(\alpha)$ where $\alpha$ is the real root of $f_\alpha(x) = x^p - m$ where $m\in \QQ$, $m\not= \pm 1$ and $m$ is $p$-power free. It is clear that the other roots are $\zeta_p^i\alpha$ for $i\in\{1,\hdots, {p-1}\}$ hence $\QQ(\zeta_p)\subseteq \tilde{K}$ and so it is the resolvent as desired.

\noindent
($\Leftarrow$) Consider $K$, $[K: \QQ] = p$, with $\Gal(\tilde{K}/\QQ)\iso F_p \iso  \mathbb{F}_p \rtimes \mathbb{F}_p^\times$ and resolvent field equal to $\QQ(\zeta_p)$. This gives us the following tower:
\[
\begin{tikzcd}
												&	\tilde{K}		\\
	\QQ(\zeta_p)\arrow[dash]{ur}{\langle \sigma \rangle}		&	K\arrow[dash]{u}		\\
	\QQ\arrow[dash]{u}{\langle \tau \rangle}\arrow[dash]{ur}
\end{tikzcd}
\]
Where $\left(\ZZ/p\ZZ\right)^\times = \langle \tau \rangle$ and $\ZZ/p\ZZ = \langle \sigma \rangle$. We will show that $K = \QQ\left(\sqrt[p]{m'}\right)$ for some $m'\in\QQ$. 
\\
\\
Fixing $g$ to be a generator of $\left(\ZZ/p\ZZ\right)^\times$; we have that
	\[	 F_p = \langle \tau, \sigma \mid \tau^{p-1} = \sigma^p = 1 \text{ and } \tau \sigma = \sigma^g \tau \rangle	\]
and using this we see that $\tau(\zeta_p) = \zeta_p^g$, $\sigma(\alpha) = \zeta_p\alpha$, $\sigma(\zeta_p) = \zeta_p$. We know by Kummer Theory that $K_p = \QQ(\zeta_p)(\sqrt[p]{m})$ for some $m\in \QQ(\zeta_p)$, so letting $\alpha = \sqrt[p]{m}$ we can consider the element 
	\[	\theta = \prod_{i=1}^{p-1} \tau^i(\alpha)  \]
It is clear that $\theta$ is fixed by $\tau$ and, as there are no intermediate extensions of $K/\QQ$, that $\theta\in K$ or $\theta \in \QQ$. Showing that $\sigma$ acts non-trivially on $\theta$ will show that $\theta\in K\setminus \QQ$ and hence $K = \QQ(\theta)$.
The condition that $\tau\sigma = \sigma^g \tau$ is equivalent to $\sigma\tau = \tau \sigma^{g^{-1}}$, where $g^{-1}\in \ZZ$ is such that $g^{-1}g \equiv 1 \mod{p}$,
and using this we see that:
	\[	\sigma\tau^i = \tau^i \sigma^{g^{-i}}	\]
Using that $\sigma$ acts on $\alpha$ via multiplication by $\zeta_p$, and acts trivially on $\zeta_p$, we have that:
	\begin{align*}
		\sigma\tau^i(\alpha)	&	= 	\tau^i\sigma^{g^{-i}}(\alpha)	\\
						&	=	\tau^i (\zeta_p^{g^{-i}}\alpha)	\\
						&	=	\zeta^{g^i\cdot g^{-i}} \tau^i(\alpha)	\\
						&	=	\zeta\tau^i(\alpha)
	\end{align*}
and this implies that:
	\begin{align*}
		\sigma(\theta)	&	=	\sigma\left(\prod_{i=1}^{p-1} \tau^i(\alpha)\right)	\\
					&	=	\prod_{i=1}^{p-1} \sigma\tau^i(\alpha)	\\
					&	=	\prod_{i=1}^{p-1} \zeta\tau^i(\alpha)		\\
					&	=	\zeta^{p-1}\theta.					\\
					&	\not= \theta
	\end{align*}	
Thus, letting $m' = \prod_{i=1}^{p-1} \tau^i(m)\in \QQ$ we have that $(m') = \theta^p$ and so $K=\QQ\left(\sqrt[p]{m'}\right)$ as desired. 
\end{proof}

\begin{remark}
We note that this result is widely stated as fact in the case of cubic number fields where we have the more familiar notion of resolvent field, however we were unable to find a general statement of this result. 
\end{remark}

\subsection{Conclusion}

Theorem \ref{theoremgalois} allows us to phrase this study in terms of Galois conditions and resolvent fields, rather than pure prime degree number fields. Our desire to reframe the study in this way is partially motivated by the work of Cohen and Thorne who count, for example, $D_p$ extensions with a fixed quadratic resolvent \cite{CohenThorne}. In turn, it is motivated by Malle's conjecture and all the work that is being done towards counting number fields with prescribed Galois group. Using this, we are able to state all the results of this paper as results about shapes and (regularized) equidistribution of $F_p$ number fields with resolvent field $\QQ(\zeta_p)$. Furthermore, the work in this paper can then be used to show that the number of $F_p$ number fields with fixed resolvent $\QQ(\zeta_p)$, and discriminant bounded by $X$, grows like:
	\[	 X^{\frac{1}{p-1}}\log^{p-2}(X)	\]
Of course the number of such fields seems quite minimal given that they arise as pure extensions but this begs the question, which the author hopes to answer in future work: 

\begin{center} \emph{{\bf Question:} How does this study compare to those $F_p$ number fields, $K$, with resolvent $K_{p-1}\not=\QQ(\zeta_p)$?} \end{center}

This is motivated by the result mentioned in \ref{cubicresolvent} regarding the asymptotics of cubic number fields with a fixed quadratic resolvent, $K_2= \QQ(\sqrt{d})$.
Of course when $d=-3$ we have that $K_2 = \QQ(\zeta_3) $ and so what one might expect is that the asymptotic for Frobenius fields with resolvent $K_{p-1}\not=\QQ(\zeta_p)$ has fewer $\log$ terms.

\section*{Acknowledgments}

The author would like to thank Rob Harron and Ila Varma for their guidance and continued support during all phases of this project. Thanks to Piper H for many helpful conversations and Pavel Guerzhoy, Piper H, Michelle Manes, Khoa Nguyen, Ari Shnidman and Christelle Vincent for their feedback on earlier versions of this work. Special thanks go to Henri Cohen, who answered a question of the author on mathoverflow regarding a determinant formula that appears in the proof of \cref{jacobiandeterminant}, and to Vlad Matei who referred the author to the work of Carlitz and Olson, in \cite{carlitzolson}.

\bibliography{mybib}

\end{document}